\documentclass{amsart}
\usepackage{longtable}
\usepackage{amsfonts,amsmath,amsthm,amscd,amssymb,latexsym,amsbsy,pb-diagram,caption,url}
\usepackage[noadjust]{cite}
\usepackage{tikz}
\usetikzlibrary{positioning,trees, decorations.markings}
\usepackage[english]{babel}

\newtheorem{theorem}{Theorem}[section]
\newtheorem{corollary}[theorem]{Corollary}
\newtheorem{lemma}[theorem]{Lemma}
\newtheorem{proposition}[theorem]{Proposition}

\newtheorem{example}[theorem]{Example}
\newtheorem{problem}{Problem}

\theoremstyle{definition}
\newtheorem{definition}[theorem]{Definition}

\theoremstyle{remark}
\newtheorem{remark}[theorem]{Remark}
\numberwithin{equation}{section}

\DeclareMathOperator{\diam}{diam}

\numberwithin{equation}{section}

\makeatletter

\renewcommand{\p@enumii}{}
\makeatother

\begin{document}

\title{Isomorphism of trees and isometry of ultrametric spaces}
\author{Oleksiy Dovgoshey}
\address{Functions Theory Department, Institute of Applied Mathematics and Mechanics of NASU, Dobrovolskogo str. 1, Slovyansk 84100, Ukraine}
\email{oleksiy.dovgoshey@gmail.com}

\begin{abstract}
We study the conditions under which the isometry of spaces with metrics generated by weights given on the edges of finite trees is equivalent to the isomorphism of these trees. Similar questions are studied for ultrametric spaces generated by labelings given on the vertices of trees. The obtained results generalized some facts previously known for phylogenetic  trees and for Gurvich---Vyalyi monotone trees.
\end{abstract}



\maketitle

\section{Introduction}

In 2001 at the Workshop on General Algebra the attention of experts on the theory of lattices was paid to the following problem of I.~M.~Gelfand: Using graph theory describe up to isometry all finite ultrametric spaces~\cite{Lem2001}. An appropriate representation of ultrametric spaces by monotone rooted trees was proposed by V.~Gurvich and M.~Vyalyi in~\cite{GV2012DAM}. A simple geometric description of Gurvich---Vyalyi representing trees was found in \cite{PD2014JMS}. This description allows us effectively use the Gurvich---Vyalyi representation in various problems associated with finite ultrametric spaces. In particular, this leads to a graph-theoretic interpretation of the Gomory-Hu Inequality \cite{DPT2015}. A characterization of finite ultrametric spaces which are as rigid as possible also was obtained \cite{DPT2017FPTA} on the basis of the Gurvich---Vyalyi representation. Some other extremal properties of finite ultrametric spaces and related them properties of monotone rooted trees have been found in~\cite{DP2020pNUAA}. The interconnections between the Gurvich---Vyalyi representation and the space of balls endowed with the Hausdorff metric are discussed in~\cite{Dov2019pNUAA} (see also \cite{Qiu2009pNUAA, Qiu2014pNUAA, DP2018pNUAA, Pet2018pNUAA, Pet2013PoI}).

It is well-known that the sets of leaves of phylogenetic equidistant trees with additive metric are ultrametric. The finite equidistant trees can be considered as finite subtrees of the so-called \(R\)-trees (see~\cite{Ber2019SMJ} for some interesting results related to \(R\)-trees and ultrametrics). The categorical equivalence of trees and ultrametric spaces was investigated in \cite{H04} and \cite{Lem2003AU}. 

It is interesting to note that, in fact, both the monotone trees and equidistant trees were used in the phylogenetic for the description of related ultrametric spaces long before the publication of paper \cite{GV2012DAM} (see, for example, \cite{SS2003OUP}). The monotone trees and the equidistant trees are dual in a certain sense, but it seems that the description of this duality can be found in Section~7.1 of book~\cite{SS2003OUP} only. 

In the present paper we discuss the interrelations between the weighted trees with additive metrics and labeled trees with corresponding ultrametrics. In particular, the duality of equidistant trees and monotone trees is studied in details for trees which are more general than the classical phylogenetic trees.

The paper is organized as follows. 

Section~\ref{sec2} contains some standard definitions from the theory of metric spaces and the theory of graphs. A short list of known properties of Gurvich---Vyalyi representing trees is also given there.

Section~\ref{sec3} deals with the weights and labelings on unrooted trees. In Theorem~\ref{t3.4} we prove that, for all weights, the isomorphisms of weighted trees coincide with isometries of metric spaces endowed with the corresponding additive metrics and, in Proposition~\ref{p3.4}, we show that this property characterizes the trees among all connected weighted graphs. Proposition~\ref{p3.5} describes conditions under which the labelings on the vertex sets of trees generate ultrametrics. In Theorem~\ref{t3.9} it is shown that under the same conditions every connected labeled graph \(G\) contains a spanning tree \(T\) such that labeling, induced on \(T\), generates the same ultrametric as an original labeling on \(G\). Theorem~\ref{t3.12}, one of the main results of the section, shows that, in the contrast with the weighted trees, the isomorphisms of labeled trees coincide with isometries of generated ultrametric spaces only for trees with one vertex.

In Section~\ref{sec4}, after defining the concepts of equidistant weight and monotone labeling for the case of an arbitrary rooted tree, we find explicit formulas describing the transition between these weights and labelings in Proposition~\ref{p6.1}. Good functorial properties of such transition are described by Proposition~\ref{p3.3}.

Theorem~\ref{t6.4}, Theorem~\ref{c3.13} and Theorem~\ref{t4.13} are generalizations of the well-known fact about representation of finite ultrametric spaces by phylogenetic trees and by Gurvich---Vyalyi monotone trees.

Proposition~\ref{p6.10} describes the necessary and sufficient conditions under which isomorphism of equidistant trees (monotone trees) is equivalent to isometricity of corresponding ultrametric spaces.

Section~\ref{sec5} of the paper contains some characteristic properties of ultrametric spaces of leaves of equidistant rooted trees with pedant roots. In particular, Proposition~\ref{p3.9} describes a new characteristic property of stars in the language of equidistant weights.

\section{Initial definitions and facts}\label{sec2}

Let us recall some concepts from the theory of metric spaces and the graph theory.

A \textit{metric} on a set $X$ is a function $d\colon X\times X\rightarrow \mathbb{R}^+$, $\mathbb R^+ = [0,\infty)$, such that for all \(x\), \(y\), \(z \in X\)
\begin{enumerate}
\item $d(x,y)=d(y,x)$,
\item $(d(x,y)=0)\Leftrightarrow (x=y)$,
\item \(d(x,y)\leq d(x,z) + d(z,y)\).
\end{enumerate}
The quantity
$$
\diam X=\sup\{d(x,y)\colon x,y\in X\}
$$
is the \emph{diameter} of $(X,d)$. 

A metric space \((X, d)\) is \emph{ultrametric} if the \emph{strong triangle inequality}
\[
d(x,y)\leq \max \{d(x,z),d(z,y)\}.
\]
holds for all \(x\), \(y\), \(z \in X\). In this case \(d\) is called \emph{an ultrametric} on \(X\) and \((X, d)\) is an \emph{ultrametric space}.

\begin{definition}\label{d1.1}
Metric spaces \((X, d)\) and \((Y, \rho)\) are \emph{isometric} if there is a bijective mapping \(\Phi \colon X \to Y\) such that
\[
d(x,y) = \rho(\Phi(x), \Phi(y))
\]
holds for all \(x\), \(y \in X\). In this case we write \((X, d) \simeq (Y, \rho)\) and say that \(\Phi\) is an \emph{isometry} of \((X, d)\) and \((Y, \rho)\).
\end{definition}

A \textit{graph} is a pair $(V,E)$ consisting of a nonempty set $V$ and a (probably empty) set $E$ whose elements are unordered pairs of different points from $V$. For a graph $G=(V,E)$, the sets $V=V(G)$ and $E=E(G)$ are called \textit{the set of vertices (or nodes)} and \textit{the set of edges}, respectively. We say that \(G\) is \emph{nonempty} if \(E(G) \neq \varnothing\). If $\{x,y\} \in E(G)$, then the vertices $x$ and $y$ are \emph{adjacent}. Recall that a \emph{path} is a nonempty graph $P$ whose vertices can be numbered so that
$$
V(P) = \{x_0,x_1,...,x_k\},\ k \geqslant 1, \quad \text{and} \quad E(P) = \{\{x_0,x_1\},...,\{x_{k-1},x_k\}\}.
$$
In this case we say that \(P\) is a path joining \(x_0\) and \(x_k\).

A graph \(G\) is \emph{finite} if \(V(G)\) is a finite set, \(|V(G)| < \infty\). In this paper we consider finite graphs only. A finite graph $C$ is a \textit{cycle} if $|V(C)|\geq 3$ and there exists an enumeration \(v_1\), \(v_2\), \(\ldots\), \(v_n\) of its vertices such that
\begin{equation*}
(\{v_i,v_j\}\in E(C))\Leftrightarrow (|i-j|=1\quad \mbox{or}\quad |i-j|=n-1).
\end{equation*}
A graph $H$ is a \emph{subgraph} of a graph $G$ if
$$
V(H) \subseteq V(G) \quad \text{and} \quad E(H) \subseteq E(G).
$$
We write $H\subseteq G$ if $H$ is a subgraph of $G$.

A graph \(G\) is \emph{connected} if for every two distinct \(u\), \(v \in V(G)\) there is a path \(P \subseteq G\) joining \(u\) and \(v\). A connected graph without cycles is called a \emph{tree}. A vertex $v$ of a tree $T$ is a \emph{leaf} if the \emph{degree} of $v$ is less than two,
$$
\delta(v) = |\{u \in V(T) \colon \{u, v\} \in E(T)\}| < 2.
$$
If a vertex $v$ is not a leaf of $T$, then we say that $v$ is an \emph{internal node} of $T$. 

A tree $T$ may have a distinguished vertex $r$ called the \emph{root}; in this case $T$ is called a \emph{rooted tree} and we write $T=T(r)$. 

Let $T = T(r)$ be a rooted tree and let $v$ be a vertex of $T$. Denote by $\delta^+(v)$ the \emph{out-degree} of $v$,
\[
\delta^+(v) = \begin{cases}
\delta(v) & \text{if } v = r,\\
\delta(v)-1 & \text{if } v \neq r.
\end{cases}
\]
The root $r$ is a leaf of $T$ if and only if $\delta^+(r) \leqslant 1$. Moreover, for a vertex $v$ different from the root $r$, the equality $\delta^+(v) = 0$ holds if and only if $v$ is leaf of $T$. 

\begin{definition}\label{d2.2}
A \emph{labeled tree} $T=T(l)$ is a tree $T$ with a \emph{labeling} \(l\colon V(T)\to \mathbb{R}^{+}\). A \emph{weighted tree} \(T = T(w)\) is a nonempty tree \(T\) with a \emph{weight} \(w \colon E(T) \to \mathbb{R}^{+}\).
\end{definition}

Thus, in what follows, we assume that the labels on the tree vertices are some nonnegative numbers.

We also use the notion of complete multipartite graph.

\begin{definition}\label{def3.1}
A nonempty graph $G$ is called \emph{complete $k$-partite} if its vertices can be divided into disjoint nonempty sets $X_1$, $\ldots$, $X_k$ so that $k \geqslant 2$ and there are no edges joining the vertices of the same set $X_i$ and any two vertices from different $X_i,X_j$, $1\leqslant i,j \leqslant k$ are adjacent. In this case we write $G=G[X_1,...,X_k]$.
\end{definition}
We shall say that $G$ is a {\it complete multipartite graph} if there exists $k$ such that $G$ is complete $k$-partite.

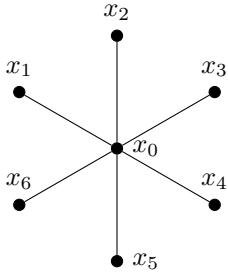
\begin{figure}[h]
\begin{center}
\begin{tikzpicture}[grow cyclic,
level 1/.style={level distance=15mm, sibling angle=60},
level 2/.style={level distance=15mm,sibling distance=12mm},
level 3/.style={level distance=15mm,sibling distance=6mm},
solid node/.style={circle,draw,inner sep=1.5,fill=black},
hollow node/.style={circle,draw,inner sep=1.5}]

\node (root) [solid node, label=right:{\(x_0\)}] at (0,0) {}
child {node(1) [solid node, label=above:{\(x_6\)}] {}}
child {node(2) [solid node, label=right:{\(x_5\)}] {}}
child {node(3) [solid node, label=above:{\(x_4\)}] {}}
child {node(3) [solid node, label=above:{\(x_3\)}] {}}
child {node(3) [solid node, label=above:{\(x_2\)}] {}}
child {node(3) [solid node, label=above:{\(x_1\)}] {}};
\end{tikzpicture}
\end{center}
\caption{A star \(G[X_1, X_2]\), \(X_1 = \{x_0\}\), \(X_2 = \{x_1, x_2, x_3, x_4, x_5, x_6\}\).}
\label{fig1.1}
\end{figure}

In particular, a \emph{star} is a complete bipartite graph \(G[X_1, X_2]\) with \(|X_1| = 1\) or \(|X_2| = 1\) \cite[p.~4]{BM}.

\begin{definition}[\cite{DDP(P-adic)}]\label{d14}
Let $(X,d)$ be a finite ultrametric space with \(|X| \geqslant 2\). Define a graph \(G_{D, X}\) as \(V(G_{D, X}) = X\) and 
\[
(\{u,v\}\in E(G_{D, X}))\Leftrightarrow (d(u,v)=\diam X)
\]
for all \(u\), \(v \in V(G_{D, X})\). We call $G_{D, X}$ the \emph{diametrical graph} of $X$.
\end{definition}

\begin{theorem}[\cite{DDP(P-adic)}]\label{t13}
Let $(X,d)$ be a finite ultrametric space, $|X|\geqslant 2$. Then $G_{D,X}$ is complete multipartite.
\end{theorem}

For every nonempty, finite ultrametric space $(X, d)$ we can associate a labeled rooted tree $T_X=T_X(r,l)$ with $r=X$ and $l\colon V(T_{X})\to \mathbb{R}^{+}$ by the following rule (see~\cite{PD2014JMS}).

If $X$ is a one-point set, then $T_X$ is the rooted tree consisting of the node $X$ with the label~$\diam X=0$. Note that for the rooted trees consisting only of one node, we consider that this node is the root as well as a leaf.

Let $|X|\geqslant 2$. According to Theorem~\ref{t13} we have $G_{D,X} = G_{D,X}[X_1,...,X_k]$ and $k \geqslant 2$. In this case the root of the tree $T_X$ is labeled by $\diam X$ and $T_X$ has the nodes $X_1,...,X_k$ of the first level with the labels
\begin{equation}\label{e2.7}
l(X_i)= \diam X_i,
\end{equation}
$i = 1,...,k$. The nodes of the first level with the label $0$ are leaves, and those indicated by strictly positive labels are internal nodes of the tree $T_X$. If the first level has no internal nodes, then the tree $T_X$ is constructed. Otherwise, by repeating the above-described procedure with $X_1,...,X_k$ instead of $X$, we obtain the nodes of the second level, etc. Since $X$ is finite, all vertices on some level will be leaves, and the construction of $T_X$ is completed.

The above-constructed labeled rooted tree $T_X$ is called the \emph{representing tree} of the ultrametric space $(X, d)$. 

In the next sections we will need several different concepts of isomorphism of trees.

\begin{definition}\label{d1.5}
Let $T_1$ and $T_2$ be trees. A bijection $f\colon V(T_1)\to V(T_2)$ is an \emph{isomorphism} of $T_1$ and $T_2$ if
\begin{equation}\label{d1.5:e1}
(\{u,v\} \in E(T_1)) \Leftrightarrow (\{f(u),f(v)\} \in E(T_2))
\end{equation}
is valid for all $u$, $v \in V(T_1)$. The trees $T_1$ and $T_2$ are \emph{isomorphic} if there exists an isomorphism $f\colon V(T_1) \to V(T_2)$.

Let $T_1 = T_1(r_1)$ and $T_2 = T_2(r_2)$ be rooted trees. Then $T_1$ and $T_2$ are \emph{isomorphic as rooted trees} if there is a bijection $f\colon V(T_1) \to V(T_2)$ such that $f(r_1) = r_2$ and \eqref{d1.5:e1} is valid for all \(u\), \(v \in V(T_1)\).
\end{definition}

For the case of labeled trees the above definition must be modified as follows.

\begin{definition}\label{d1.6}
Let $T_i=T_i(l_i)$ be labeled trees with the labelings $l_i\colon V(T_i)\to \mathbb{R}^{+}$, \(i = 1\), \(2\). An isomorphism $f\colon V(T_1) \to V(T_2)$ of the trees $T_1$ and $T_2$ is an isomorphism of the labeled trees $T_1(l_1)$ and $T_2(l_2)$ if the equality
\begin{equation}\label{d1.6e1}
l_2(f(v))=l_1(v)
\end{equation}
holds for every $v \in V(T_1)$. 

Analogously, for weighted trees \(T_1 = T_1(w_1)\) and \(T_2 = T_2(w_2)\) an isomorphism \(f \colon V(T_1) \to V(T_2)\) of \(T_1\) and \(T_2\) is an isomorphism of  \(T_1(w_1)\) and \(T_2(w_2)\) if the equality
\begin{equation}\label{d1.6e2}
w_2(\{f(u), f(v)\}) = w_1(\{u, v\})
\end{equation}
holds for every \(\{u, v\} \in E(T_1)\).
\end{definition}

The formulas
\[
T_1(l_1) \simeq T_2(l_2) \quad \text{and} \quad T_1(w_1) \simeq T_2(w_2)
\]
mean that the labeled trees \(T_1(l_1)\), \(T_2(l_2)\) are isomorphic and, respectively, that the weighted trees \(T_1(w_1)\), \(T_2(w_2)\) are isomorphic.

\begin{definition}\label{d2.8}
Let \(T_1 = T_1(r_1, l_1)\) and \(T_2 = T_2(r_2, l_2)\) be labeled rooted trees. An isomorphism \(f \colon V(T_1) \to V(T_2)\) of the free (unrooted, without labelings) trees  \(T_1\) and \(T_2\) is an isomorphism of \(T_1(r_1, l_1)\) and \(T_2(r_2, l_2)\) if it is an isomorphism of \(T_1(r_1)\) and \(T_2(r_2)\) and, simultaneously, an isomorphism of \(T_1(l_1)\) and \(T_2(l_2)\) (\(T_1(w_1)\) and \(T_2(w_2)\)).

The labeled rooted trees \(T_1(r_1, l_1)\) and \(T_2(r_2, l_2)\) are isomorphic if there is an isomorphism of these trees. In this case we write \(T_1(r_1, l_1) \simeq T_2(r_2, l_2)\).

Analogously, we define an isomorphism of weighted rooted trees \(T_1(r_1, w_1)\), \(T_2(r_2, w_2)\) and write \(T_1(r_1, w_1) \simeq T_2(r_2, w_2)\) if \(T_1(r_1, w_1)\) and \(T_2(r_2, w_2)\) are isomorphic as weighted rooted trees.
\end{definition}

\begin{theorem}[\cite{DPT2017FPTA}]\label{t2.7}
Let \((X, d)\) and \((Y, \rho)\) be nonempty, finite ultrametric spaces with the representing trees \(T_X\) and \(T_Y\).  Then 
\[
((X, d) \simeq (Y, \rho)) \Leftrightarrow (T_X \simeq T_Y)
\]
is valid.
\end{theorem}

If \(T = T(r)\) is a rooted tree and \(u\), \(v \in V(T)\) are distinct, then we say that \(v\) is a \emph{successor} of \(u\) whenever \(u \in V(P)\), where \(P\) is the path joining \(v\) and \(r\). A successor \(v\) of \(u\) is a \emph{direct successor} of \(u\) if \(\{u, v\} \in E(T)\).

The following theorem is a simple modification of Theorem~2.7 \cite{DP2019PNUAA}.

\begin{theorem}\label{t1.7}
Let $T=T(r, l)$ be a labeled rooted tree. Then the following two conditions are equivalent:
\begin{enumerate}
\item\label{t1.7:c1} For every $u \in V(T)$ we have $\delta^+(u)\neq 1$ and
$$
(\delta^+(u) =0) \Leftrightarrow (l(u)=0)
$$
and, in addition, the inequality \(l(v) < l(u)\) holds whenever $v$ is a direct successor of $u$.
\item\label{t1.7:c2} There is a finite, nonempty ultrametric space $(X,d)$ such that 
\[
T_X \simeq T(r, l).
\]
\end{enumerate}
\end{theorem}

Recall that a \emph{phylogenetic tree} is an unrooted tree \(T\), whose inner nodes have degree at least three, together with a labeling defined on the set of leaves of \(T\) (see, for example, \cite{SS2009MH}). Using Theorem~\ref{t13} and above described procedure of construction of representing trees we can simply prove the following result.

\begin{proposition}\label{p2.13}
The following statements are equivalent for every tree \(T\).
\begin{enumerate}
\item\label{p2.13:s1} There is a phylogenetic tree \(T_1\) such that \(T\) and \(T_1\) are isomorphic as free (unrooted, without labelings) trees.
\item\label{p2.13:s2} There is an ultrametric space \((X, d)\) such that the diametrical graph \(G_{D, X}\) is complete \(k\)-partite with \(k \geqslant 3\) and \(T_X\) and \(T\) are isomorphic as free (unrooted, without labelings) trees.
\end{enumerate}
\end{proposition}

Let $(X,d)$ be an ultrametric space. Recall that a \emph{ball} with a radius $r \geqslant 0$ and a center $c\in X$ is the set 
\[
B_r(c)=\{x\in X\colon d(x,c)\leqslant r\}.
\]
The \emph{ballean} $\mathbf{B}_X$ of the ultrametric space $(X,d)$ is the set of all balls of $(X,d)$. Every one-point subset of $X$ belongs to $\mathbf{B}_X$,  this is a \emph{singular} ball in~$X$.

Let $T = T(r)$ be a rooted tree. For every vertex $v$ of $T$ we denote by $T_v$ the subtree of $T$ such that \(v\) is the root of \(T_v\),
\begin{equation}\label{e2.5}
V(T_v) = \{u \in V(T) \colon u = v \text{ or \(u\) is a successor of } v\},
\end{equation}
and satisfying
\begin{equation}\label{e2.6}
(\{v_1, v_2\} \in E(T_v)) \Leftrightarrow (\{v_1, v_2\} \in E(T(r)))
\end{equation}
for all \(v_1\), \(v_2 \in V(T_v)\). In this situation Charles Semple and Mike Steel say that \(T_v\) is a rooted subtree of \(T(r)\) lying below \(v\) \cite[p.~9]{SS2003OUP}. See Figure~\ref{fig1} for an example of such a rooted subtree. 

In what follows we write \(L = L(T_v)\) to denote the set of all leaves of \(T_v\).

If \(T = T_X\) is the representing tree of a finite ultrametric space \((X, d)\) and \(v = \{x_1, x_2, \ldots, x_n\}\) is a vertex of \(T_X\), \(x_i \in X\), \(i = 1, \ldots, n\), then we have
\[
L(T_v) = \{\{x_1\}, \ldots, \{x_n\}\}.
\]

\begin{figure}[h]
\begin{center}
\begin{tikzpicture}
\tikzstyle{level 1}=[level distance=15mm,sibling distance=25mm]
\tikzstyle{level 2}=[level distance=15mm,sibling distance=12mm]
\tikzstyle{level 3}=[level distance=15mm,sibling distance=6mm]
\tikzset{
solid node/.style={circle,draw,inner sep=1.5,fill=black},
hollow node/.style={circle,draw,inner sep=1.5}
}
\node [label=left:{\(T(v_1)\)}] at (3,0) {};
\node (0) [solid node, label=above:{\(v_1\)}] at (0,0) {}
child [sibling distance=15mm] {node(1) [solid node, label=above:{\(v_2\)}] {}}
child{node(1) [solid node, label=above left:{\(v_3\)}]{}
	child{node[solid node, label=above:{\(v_5\)}]{}
		child{node[solid node, label=below:{\(v_{11}\)}] {}}
		child{node[solid node, label=below:{\(v_{12}\)}] {}}
		}
	child{node[solid node, label=above left:{\(v_6\)}]{}
		child{node[solid node, label=below:{\(v_{13}\)}]{}}
		child{node[solid node, label=below:{\(v_{14}\)}]{}}
		}
	child [sibling distance=6mm] {node[solid node, label=below:{\(v_7\)}]{}}
	child [sibling distance=6mm] {node[solid node, label=below:{\(v_8\)}]{}}
}
child{node[solid node, label=above:{\(v_4\)}] {}
child{node[solid node, label=below:{\(v_9\)}]{}}
child{node[solid node, label=below:{\(v_{10}\)}]{}}
};

\tikzstyle{level 1}=[level distance=15mm,sibling distance=12mm]
\tikzstyle{level 2}=[level distance=15mm,sibling distance=6mm]
\tikzstyle{level 3}=[level distance=15mm,sibling distance=6mm]

\node [label=left:{\(T_{v_3}\)}] at (6,-1.5) {};
\node (1) [solid node, label=above:{\(v_3\)}] at (7,-1.5) {}
child {node[solid node, label=above:{\(v_5\)}]{}
	child{node[solid node, label=below:{\(v_{11}\)}] {}}
	child{node[solid node, label=below:{\(v_{12}\)}] {}}
	}
child {node[solid node, label=above left:{\(v_6\)}]{}
	child{node[solid node, label=below:{\(v_{13}\)}]{}}
	child{node[solid node, label=below:{\(v_{14}\)}]{}}
	}
child [sibling distance=6mm] {node[solid node, label=below:{\(v_7\)}]{}}
child [sibling distance=6mm] {node[solid node, label=below:{\(v_8\)}]{}};
\end{tikzpicture}
\end{center}
\caption{The rooted tree \(T_{v_3}\) is a rooted subtree of the rooted tree \(T(v_1)\) lying below \(v_3\). Here \(L(T_{v_3}) = \{v_7, v_8, v_{11}, v_{12}, v_{13}, v_{14}\}\).}
\label{fig1}
\end{figure}
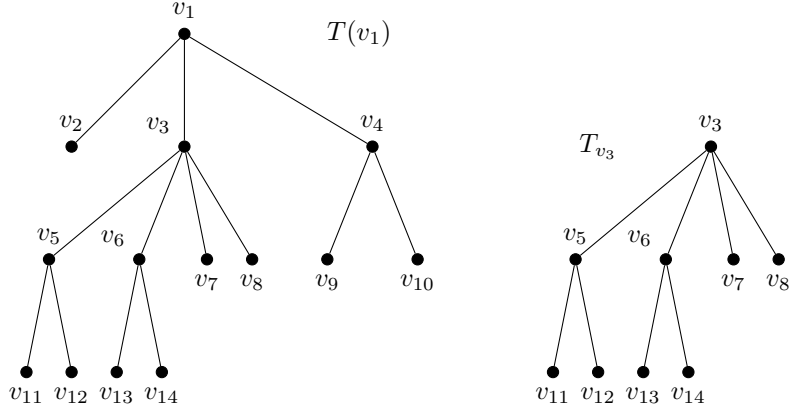

For $T_X$ consisting of one node only, \(V(T_X) = X\), we consider that \(X\) is the root of $T_X$ as well as a leaf of \(T_X\). Thus, if \(X = \{x\}\), then we have \(L(T_X) = \{\{x\}\}\).

The following proposition claims that the ballean of every finite ultrametric space $(X,d)$ coincides with the vertex set of $T_X$.

\begin{proposition}\label{p2.8}
Let $(X,d)$ be a finite, nonempty ultrametric space with representing tree $T_X$. Then the following statements hold:
\begin{enumerate}
\item For every node $v\in V(T_X)$ there is the unique \(B \in \mathbf{B}_X\), \(B = \{x_1, \ldots, x_n\}\), such that \(L(T_v) = \{\{x_1\}, \ldots, \{x_n\}\}\). This ball \(B\) is singular if and only if \(v \in L(T_X)\).
\item For every $B \in \mathbf{B}_X$, \(B = \{x_1, \ldots, x_n\}\), there is the unique $v \in V(T_X)$ such that $L(T_v) = \{\{x_1\}, \ldots, \{x_n\}\}$. This vertex \(v\) is a leaf of \(T\) if and only if \(B\) is singular.
\end{enumerate}
\end{proposition}

This proposition is a simple modification of the corresponding result from~\cite{Pet2013PoI} (see also Theorem~2.5 in \cite{Dov2019pNUAA}).

The proofs of Theorem~\ref{t2.7}, Theorem~\ref{t1.7} and Proposition~\ref{p2.8} are based on the next basic theorem.

\begin{theorem}[\cite{DP2018pNUAA}]\label{t2.9}
Let \((X, d)\) be a finite ultrametric space and let \(x_1\) and \(x_2\) be two distinct points of \(X\). If \(P\) is the path joining the distinct leaves \(\{x_1\}\) and \(\{x_2\}\) in \(T_X\), then we have
\[
d(x_1, x_2) = \max_{v \in V(P)} l(v),
\]
where the labeling \(l \colon V(T_X) \to \mathbb{R}^{+}\) is defined as in \eqref{e2.7}.
\end{theorem}

The next section of the paper contains several results which can be considered as extensions of Theorem~\ref{t2.9} to the case of unrooted, labeled trees.

\section*{Comments on the concept of graph isomorphism}

All above introduced notions of tree isomorphism can be considered as some specifications of the following general definition.

Let \(G = (V, E, L_V, L_E, l)\) be a graph with \(V(G) = V\), \(E(G) = E\) and such that \(L_V\) and \(L_E\) are some sets of vertex labels and edge labels, respectively, and let \(f \colon (V \cup E) \to L_V \cup L_E\) be a mapping for which \(f(V) \subseteq L_V\) and \(f(E) \subseteq L_{E}\) hold.

\begin{definition}\label{d2.11}
The graphs 
\[
G_1 = (V_1, E_1, L_{V_1}, L_{E_1}, l_1) \quad \text{and} \quad G_2 = (V_2, E_2, L_{V_2}, L_{E_2}, l_2)
\]
are isomorphic if there is a bijection \(f \colon V_1 \to V_2\) such that:
\begin{enumerate}
\item\label{d2.11:s1} \(l_1(v) = l_2(f(v))\) for every \(v \in V_1\).
\item\label{d2.11:s2} \((\{u, v\} \in E_1) \Leftrightarrow (\{f(u), f(v)\} \in E_2)\) for all \(u\), \(v \in V_1\).
\item\label{d2.11:s3} \(l_1(\{u, v\}) = l_2(\{f(u), f(v)\})\) for every \(\{u, v\} \in E_1\)
\end{enumerate}
(See Definition~1 in \cite{HHH2006}).
\end{definition}

Two definitions of graph isomorphism are equivalent if any two graphs are isomorphic w.r.t the first definition if and only if these graphs are isomorphic w.r.t. the second one. For the case when 
\[
G_1 = (V_1, E_1, L_{V_1}, L_{E_1}, l_1) \quad \text{and} \quad G_2 = (V_2, E_2, L_{V_2}, L_{E_2}, l_2)
\]
are trees, Definition~\ref{d2.11} is equivalent to:
\begin{itemize}
\item Definition~\ref{d1.5} if \(L_{V_1} \cup L_{E_1} = L_{V_2} \cup L_{E_2} = \{0\}\).
\item The definition of isomorphic rooted trees if \(L_{V_1} \cup L_{E_1} = L_{V_2} \cup L_{E_2} = \{0, 1\}\) and there are the unique \(v^{*}_1 \in V_1\) and \(v^{*}_2 \in V_2\) such that \(l_i(v^{*}_i) = 1\) and \(l_i(e) = l_i(v) = 0\) whenever \(e \in E_i\) and \(v \in V_i \setminus \{v^{*}_i\}\), \(i = 1\), \(2\).
\item Definition~\ref{d1.6} if \(L_{V_1} = L_{V_2} = \mathbb{R}^{+}\) and \(L_{E_1} = L_{E_2} = \{0\}\).
\end{itemize}

Similarly we can find specifications of Definition~\ref{d2.11} which are equivalent to definitions of isomorphisms of phylogenetic trees, weighted trees, weighted rooted trees, and weighted labeled trees, but we prefer to use several independent definitions of isomorphism to simplify the statements of future results.

The concept of equivalent isomorphisms (= equivalent definitions of isomorphism) can be more satisfactory described on the basis of category theory, but this is not the subject of the paper.

\section{Weights and labelings on unrooted trees}\label{sec3}

Let \(G(w)\) be a weighted graph. The weight \(w \colon E(G) \to \mathbb{R}^{+}\) is \emph{strictly positive} if \(w(e) > 0\) holds for every \(e \in E(G)\).

Let $T = T(w)$ be a weighted tree with the strictly positive weight $w$. The function $d_{w} \colon V(T) \times V(T) \to \mathbb{R}^{+}$ defined as
\begin{equation}\label{e6.7}
d_{w}(u, v) = \begin{cases}
0, & \text{if $u = v$}\\
\sum_{e \in E(P)} w(e), & \text{if $u \neq v$},
\end{cases}
\end{equation}
where $P$ is the unique path joining $u$ and $v$ in $T$, is a metric on $V(T)$. 

If \(G = G(w)\) is a connected, weighted graph with the strictly positive weight \(w \colon E(G) \to \mathbb{R}^{+}\), then the weighted \emph{shortest-path metric} is the mapping \(\rho_w \colon V(G) \times V(G) \to \mathbb{R}^{+}\) defined as
\begin{equation}\label{e3.2}
\rho_w(u, v) = \begin{cases}
0, & \text{if } u = v\\
\min_{P \in \mathcal{P}_{u, v}} \sum_{e \in P} w(e), & \text{if } u \neq v,
\end{cases}
\end{equation}
where \(\mathcal{P}_{u, v}\) is the set of all paths joining \(u\) and \(v\) in \(G\). If \(G\) is a tree, then, for any pair of distinct \(u\), \(v \in G\), the set \(\mathcal{P}_{u, v}\) contains the unique path. Thus, the metrics \(d_w\) and \(\rho_w\) coincide for trees.

The metric defined by~\eqref{e3.2} is called \emph{additive}. If \(w(e) = 1\) holds for every \(e \in E(G)\), then \(d_{w}\) is called the \emph{graph metric} on \(G\).

It is easy to prove that the connected graphs \(G_1\) and \(G_2\) are isomorphic as free, unweighted graphs if and only if the metric spaces \((V(G_1), d_1)\) and \((V(G_2), d_2)\) are isometric, where \(d_i\) is graph metric on \(G_i\), \(i = 1\), \(2\). The following theorem is a partial generalization of this fact. 

\begin{theorem}\label{t3.4}
Let \(T_1 = T_1(w_1)\) and \(T_2 = T_2(w_2)\) be weighted trees with strictly positive weights. Then a mapping \(f \colon V(T_1) \to V(T_2)\) is an isometry of the metric spaces \((V(T_1), d_{w_1})\) and \((V(T_2), d_{w_2})\) if and only if it is an isomorphism of the weighted trees \(T_1(w_1)\) and \(T_2(w_2)\).
\end{theorem}

\begin{proof}
Let \(u\), \(v \in V(T_1)\). If \(f\) is an isomorphism of \(T_1(w_1)\) and \(T_2(w_2)\), then, using \eqref{d1.6e2}, \eqref{e6.7} and the uniqueness of the path \(P_{u, v}\) joining \(u\) and \(v\) in \(T_1\), we obtain the equality
\[
d_{w_1} (u, v) = d_{w_2}(f(u), f(v)).
\]
Hence, \(f\) is an isometry of the metric spaces \((V(T_1), d_{w_1})\) and \((V(T_2), d_{w_2})\).

Conversely, let \(f\) be an isometry of \((V(T_1), d_{w_1})\) and \((V(T_2), d_{w_2})\). Then it is easy to see that \(f\) is an isomorphism of weighted trees \(T_1(w_1)\) and \(T_2(w_2)\) if and only if \(f\) is an isomorphism of free (unweighted) trees \(T_1\) and \(T_2\). To prove that \(f\) is an isomorphism of \(T_1\) and \(T_2\) it suffices to show that the implication
\begin{equation}\label{t3.4:e1}
(\{u_1, v_1\} \in E(T_1)) \Rightarrow (\{f(u_1), f(v_1)\} \in E(T_2))
\end{equation}
is valid for all \(u_1\), \(v_1 \in V(T_1)\). Indeed, \eqref{t3.4:e1} implies 
\[
E(T_2) \supseteq \{\{f(u_1), f(v_1)\} \colon \{u_1, v_1\} \in E(T_1)\}. 
\]
Since \(f\) is bijective, we have \(|V(T_1)| = |V(T_2)|\). Moreover, since \(T_1\) and \(T_2\) are trees, the equalities \(|V(T_1)| = |E(T_1)| + 1\) and \(|V(T_2)| = |E(T_2)| + 1\) (see, for example, \cite[Corollary~1.53]{Die2005}) hold. Consequently, we have
\begin{equation}\label{t3.4:e2}
|E(T_1)| = |E(T_2)|.
\end{equation}
Now from \eqref{t3.4:e1} and \eqref{t3.4:e2} it follows that
\[
(\{u_1, v_1\} \in E(T_1)) \Leftrightarrow (\{f(u_1), f(v_1)\} \in E(T_2)).
\]

Let us prove \eqref{t3.4:e1}. Suppose contrary that \(\{f(u_1), f(v_1)\} \notin E(T_2)\). Then there exists a path \(P^{2} = (v_1^{*}, v_2^{*}, \ldots, v_n^{*})\) such that \(P^{2} \subseteq T_2\), \(n \geqslant 3\), and \(v_1^{*} = f(u_1)\), \(v_n^{*} = f(v_1)\). By~\eqref{e6.7}, we have 
\begin{equation}\label{t3.4:e3}
d_{w_2}(v_1^{*}, v_n^{*}) = \sum_{i=1}^{n-1} w_2(\{v_i^*, v_{i+1}^{*}\}) = \sum_{i=1}^{n-1} d_{w_2} (v_i^{*}, v_{i+1}^{*}).
\end{equation}
Since \(f\) is an isometry of \((V(T_1), d_{w_1})\) and \((V(T_2), d_{w_2})\), from \eqref{t3.4:e3} it follows that
\begin{equation}\label{t3.4:e4}
\begin{aligned}
d_{w_1}(u_1, v_1) &= d_{w_1}(u_1, f^{-1}(v_2^*)) + d_{w_1}(f^{-1}(v_{n-1}^*), v_1) \\
& \quad + \sum_{i=2}^{n-2} d_{w_1}(f^{-1}(v_{i}^*), f^{-1}(v_{i+1}^*)).
\end{aligned}
\end{equation}
(For \(i = 2\), \(\ldots\), \(n-2\) we identify the one-point set \(f^{-1}(v_{i}^*)\) with the unique point of this set.)

Since \(d_{w_1}\) is an additive metric, there are 
\begin{enumerate}
\item [] a path \(P_{1}^{1} \subseteq T_1\) joining \(f^{-1}(v_1^*) = u_1\) and \(f^{-1}(v_2^*)\),
\item [] a path \(P_{2}^{1} \subseteq T_1\) joining \(f^{-1}(v_2^*)\) and \(f^{-1}(v_3^*)\),
\item [] \ldots,
\item [] and a path \(P_{n-1}^{1} \subseteq T_1\) joining \(f^{-1}(v_{n-1}^*)\) and \(f^{-1}(v_{n}^*) = v_1\).
\end{enumerate}
Using~\eqref{e6.7} again we can write \eqref{t3.4:e4} as
\begin{equation}\label{t3.4:e5}
w_1(\{u_1, v_1\}) = \sum_{i=1}^{n-1} \sum_{e \in E(P_i^{1})} w(e).
\end{equation}

Let us consider a graph \(G\) such that 
\[
V(G) = \bigcup_{i=1}^{n-1} V(P_i^{1}) \quad \text{and} \quad E(G) = \bigcup_{i=1}^{n-1} E(P_i^{1}).
\]
It is clear that \(G\) is connected, \(G \subseteq T_1\), and \(u_1\), \(v_1 \in V(G)\). Since \(T_1\) is a tree, \(G\) is a subtree of \(T_1\). Consequently, there is a unique path \(P_{u, v}\) joining \(u_1\) and \(v_1\) in \(G\). From \(G \subseteq T_1\) it follows that \(P_{u, v}\) is a path joining \(u_1\) and \(v_1\) in \(T_1\). But the unique path joining \(u_1\) and \(v_1\) in \(T_1\) is \((\{u_1, v_1\}\). Thus, \(\{u_1, v_1\} \in E(G)\). Since \(w_1(e) > 0\) holds for every \(e \in E(T_1)\), equality \eqref{t3.4:e5} implies \(E(G) = \{\{u_1, v_1\}\}\). The last equality contradicts the definition of the path \(P^{2}\).

The validity of \eqref{t3.4:e1} follows.
\end{proof}

\begin{corollary}\label{c3.1}
Let \((X, d)\) be a finite metric space. Suppose that \(T_1 = T_1(w_1)\) and \(T_2 = T_2(w_2)\) are weighted trees with strictly positive weights such that 
\[
X = V(T_1) = V(T_2).
\]
Then the equalities \(d = d_{w_1} = d_{w_2}\) imply \(T_1(w_1) = T_2(w_2)\).
\end{corollary}

\begin{proof}
It follows from Theorem~\ref{t3.4} with \(f\) satisfying \(f(x) = x\) for every \(x \in X\).
\end{proof}

\begin{corollary}\label{c3.2}
Let \(T_1 = T_1(w_1)\) and \(T_2 = T_2(w_2)\) be weighted trees with strictly positive weights. Then the equivalence
\begin{equation}\label{c3.2:e1}
(T_1(w_1) \simeq T_2(w_2)) \Leftrightarrow ((V(T_1), d_{w_1}) \simeq (V(T_2), d_{w_2}))
\end{equation}
is valid.
\end{corollary}

The concept of isomorphism of weighted trees can be naturally extended to the concept of isomorphism of weighted graphs. Let \(G_1 = G_1(w_1)\) and \(G_2 = G_2(w_2)\) be connected, nonempty graphs with the weights \(w_1 \colon E(G_1) \to \mathbb{R}^{+}\) and \(w_2 \colon E(G_2) \to \mathbb{R}^{+}\). We write \(G_1(w_1) \simeq G_2(w_2)\) and say that \(G_1(w_1)\) and \(G_2(w_2)\) are isomorphic if there is a bijection \(f \colon V(G_1) \to V(G_2)\) such that
\[
(\{u, v\} \in E(G_1)) \Leftrightarrow (\{f(u), f(v)\} \in E(G_2))
\]
for all \(u\), \(v \in V(G_1)\) and, moreover, 
\[
w_1(\{u, v\}) = w_2(\{f(u), f(v)\})
\]
holds whenever \(\{u, v\} \in E(G_1)\).

The following result shows that the validity of \eqref{c3.2:e1} is a characteristic property of trees.

\begin{proposition}\label{p3.4}
Let \(G_1\) and \(G_2\) be connected, nonempty graphs. Then the following statements are equivalent:
\begin{enumerate}
\item \label{p3.4:s1} \(G_1\) and \(G_2\) are trees.
\item \label{p3.4:s2} The equivalence
\[
(G_1(w_1) \simeq G_2(w_2)) \Leftrightarrow ((V(G_1), \rho_{w_1}) \simeq (V(G_2), \rho_{w_2}))
\]
is valid for all strictly positive \(w_1 \colon E(G_1) \to \mathbb{R}^{+}\) and \(w_2 \colon E(G_2) \to \mathbb{R}^{+}\).
\end{enumerate}
\end{proposition}

\begin{proof}
\(\ref{p3.4:s1} \Rightarrow \ref{p3.4:s2}\). This implication is valid by Corollary~\ref{c3.2}.

\(\ref{p3.4:s2} \Rightarrow \ref{p3.4:s1}\). As in the proof of Theorem~\ref{t3.4}, it is easy to see that the implication
\[
(G_1(w_1) \simeq G_2(w_2)) \Rightarrow ((V(G_1), \rho_{w_1}) \simeq (V(G_2), \rho_{w_2}))
\]
is valid. But, in general, the converse implication is false. 

Indeed, suppose that \(G\) is a finite, connected, nonempty graph which contains a cycle \(C\), \(C \subseteq G\). Let \(e_0\) be a fixed edge of \(C\), \(e_0 \in E(C)\), and let \(S_1 := |E(G)|\). Write \(G_1 = G\) and \(G = G_2\). We can define two weights \(w_1 \colon E(G_1) \to \mathbb{R}^{+}\) and \(w_2 \colon E(G_2) \to \mathbb{R}^{+}\) such that
\begin{equation}\label{p3.4:e1}
w_1(e_0) = 1 + |E(G_1)|, \quad w_2(e_0) = 2 + |E(G_2)|
\end{equation}
and \(w_1(e) = w_2(e) = 1\) whenever \(e \in E(G)\) but \(e \neq e_0\). Since
\[
w_1(E(G)) = \{1, 1 + |E(G_1)|\} \neq \{1, 2 + |E(G_1)|\} = w_2(E(G))
\]
holds, the weighted graphs \(G_1(w_1)\) and \(G_2(w_2)\) cannot be isomorphic. Now using \eqref{e3.2} and \eqref{p3.4:e1} we see that both the metric spaces \((V(G_1), \rho_{w_1})\) and \((V(G_2), \rho_{w_2})\) are isometric to the metric space \((V(G\setminus e_0), d)\), where \(G\setminus e_0\) denotes the connected graph obtained from \(G\) by deleting \(e_0\) and \(d\) is the graph metric on \(V(G\setminus e_0)\).
\end{proof}

\begin{remark}\label{r3.5}
The characterization of trees obtained in Proposition~\ref{p3.4} is similar to the characterizations of trees which were given in Corollary~3.6 of~\cite{DMV2013AC} and Corollary~5 of~\cite{DP2013SM}.
\end{remark}

Theorem~\ref{t3.4} and Proposition~\ref{p3.4} give rise the following problem:

\begin{problem}\label{pr1}
Let \(G_1(w_1)\) and \(G_2(w_2)\) be connected, weighted graphs with the strictly positive weights. Find conditions under which every isometry \(f \colon V(G_1) \to V(G_2)\) of the metric spaces \((V(G_1), \rho_{w_1})\) and \((V(G_2), \rho_{w_2})\) is an isomorphism of \(G_1(w_1)\) and \(G_2(w_2)\).
\end{problem}

In the rest of the section we consider the labeled, unrooted trees and related them ultrametric spaces and find some modifications of Theorem~\ref{t2.9}, Theorem~\ref{t3.4}, Proposition~\ref{p3.4} and Problem~\ref{pr1} in this case.

Let \(T = T(l)\) be a labeled tree with the labeling \(l \colon E(T) \to \mathbb{R}^{+}\) and let \(d_l \colon V(T) \times V(T) \to \mathbb{R}^{+}\) be a mapping defined as
\begin{equation}\label{e3.11}
d_l(u, v) = \begin{cases}
0, & \text{if } u = v\\
\max\limits_{v^{*} \in V(P)} l(v^{*}), & \text{if } u \neq v,
\end{cases}
\end{equation}
where \(P\) is the unique path joining \(u\) and \(v\) in \(T\) (cf. Theorem~\ref{t2.9}).

\begin{proposition}\label{p3.5}
The following statements are equivalent for every labeled tree \(T = T(l)\).
\begin{enumerate}
\item\label{p3.5:s1} The function \(d_l\) is an ultrametric on \(V(T)\).
\item\label{p3.5:s2} The function \(d_l\) is a metric on \(V(T)\).
\item\label{p3.5:s3} The inequality
\begin{equation}\label{p3.5:e1}
\max\{l(u_1), l(v_1)\} > 0
\end{equation}
holds for every \(\{u_1, v_1\} \in E(T)\).
\end{enumerate}
\end{proposition}

\begin{proof}
\(\ref{p3.5:s1} \Rightarrow \ref{p3.5:s2}\). It follows directly from the definitions of metrics and ultrametrics.

\(\ref{p3.5:s2} \Rightarrow \ref{p3.5:s3}\). Let \ref{p3.5:s2} hold and let \(\{u_1, v_1\} \in E(T)\). Then \(u_1 \neq v_1\) holds because \(T\) has no loops. Since \(d_l\) is a metric, from \(u_1 \neq v_1\) it follows that 
\begin{equation}\label{p3.5:e2}
d_l(u_1, v_1) > 0.
\end{equation}
Since \(P = (u_1, v_1)\) is the unique path joining \(u_1\) and \(v_1\) in \(T\), by~\eqref{e3.11} we have
\[
d_l(u_1, v_1) = \max_{v \in V(P)} l(v) = \max\{l(u_1), l(v_1)\}.
\]
Now using \eqref{p3.5:e2} we obtain~\eqref{p3.5:e1}. 

\(\ref{p3.5:s3} \Rightarrow \ref{p3.5:s1}\). Let \ref{p3.5:s3} hold. It is clear that \(d_l\) is symmetric and nonnegative. Statement \ref{p3.5:s3} and the definition of \(d_l\) imply that the inequality \(d_l(u, v) > 0\) holds for every pairs of distinct \(u\), \(v \in V(T)\). Thus, to complete the proof of validity of \ref{p3.5:s1}, it suffices to show that the strong triangle inequality
\begin{equation}\label{p3.5:e3}
d_l(v_1, v_2) \leqslant \max\{d_l(v_1, v_3), d_l(v_3, v_2)\}
\end{equation}
holds for all \(v_1\), \(v_2\), \(v_3 \in V(T)\). 

Let \(v_1\), \(v_2\), \(v_3 \in V(T)\). It is easy to see \eqref{p3.5:e3} holds if we have \(v_i = v_j\) for some different \(i\), \(j \in \{1,2,3\}\). Suppose that \(v_1 \neq v_2 \neq v_3 \neq v_1\). Write \(P_{i, j}\) for the path joining \(v_i\) and \(v_j\) in \(T\). \(1 \leqslant i < j \leqslant 3\). From~\eqref{e3.11} it follows that
\begin{equation}\label{p3.5:e4}
\max\{d_l(v_1, v_3), d_l(v_3, v_2)\} = \max\{l(v) \colon v \in V(P_{1,3}) \cup V(P_{3,2})\}.
\end{equation}
Let
\[
T_{1,2,3} = T[V(P_{1,3}) \cup V(P_{3,2})] 
\]
be the subgraph of \(T\) induced by \(V(P_{1,3}) \cup V(P_{3,2})\). Since \(v_1\), \(v_2 \in T_{1,2,3}\) and \(T_{1,2,3}\) is connected, the uniqueness of path joining \(v_1\) and \(v_2\) implies \(P_{1,2} \subseteq T_{1,2,3}\). Hence, \(V(P_{1,2}) \subseteq V(T_{1,2,3})\) holds. Consequently
\begin{equation}\label{p3.5:e5}
\max\{l(v) \colon v \in V(P_{1,2})\} \leqslant \max\{l(v) \colon v \in V(T_{1,2,3})\}.
\end{equation}
The equality 
\[
d_l(v_1, v_2) = \max\{l(v) \colon v \in V(P_{1,2})\},
\]
\eqref{p3.5:e4} and \eqref{p3.5:e5} imply \eqref{p3.5:e3}.
\end{proof}

\begin{remark}\label{r3.6}
If \(T = T(l)\) is a labeled tree but there is \(\{u_1, v_1\} \in E(T)\) such that \(l(u_1) = l(v_1) = 0\), then the mapping \(d_l \colon V(T) \times V(T) \to \mathbb{R}^{+}\) is an pseudoultrametric on \(V(T)\), i.e., \(d_l\) is a symmetric, nonnegative mapping satisfying the strong triangle inequality.
\end{remark}

\begin{corollary}\label{c3.6}
Let \(T_1 = T_1(l_1)\) and \(T_2 = T_2(l_2)\) be labeled trees such that
\begin{equation}\label{c3.6:e1}
\max\{l_i(u_i), l_i(v_i)\} > 0
\end{equation}
for every \(\{u_i, v_i\} \in E(T_i)\), \(i = 1\), \(2\). Then the implication
\begin{equation}\label{c3.6:e2}
(T_1(l_1) \simeq T_2(l_2)) \Rightarrow (V(T_1, d_{l_1}) \simeq V(T_2, d_{l_2}))
\end{equation}
is valid.
\end{corollary}

\begin{proof}
By Proposition~\ref{p3.5}, inequality \eqref{c3.6:e1} implies that \(d_{l_1}\) and \(d_{l_2}\) are ultrametrics. Hence, the right hand side in \eqref{c3.6:e1} is defined correctly. Now \eqref{c3.6:e2} follows directly from Definition~\ref{d1.6}, formula~\eqref{e3.11}.
\end{proof}

The following example shows that the converse implication
\[
(V(T_1, d_{l_1}) \simeq V(T_2, d_{l_2})) \Rightarrow (T_1(l_1) \simeq T_2(l_2))
\]
is false in general.

\begin{example}\label{ex3.7}
Let \(a_1\), \(\ldots\), \(a_5\) be a sequence of real numbers such that
\[
0 = a_1 < a_2 \leqslant \ldots \leqslant a_5.
\]
Then the ultrametric spaces \(V(T_1, d_{l_1})\) and \(V(T_2, d_{l_2})\) are isometric for the labeled trees \(T_1 = T_1(l_1)\) and \(T_2 = T_2(l_2)\) depicted by Figure~\ref{fig2}.
\end{example}

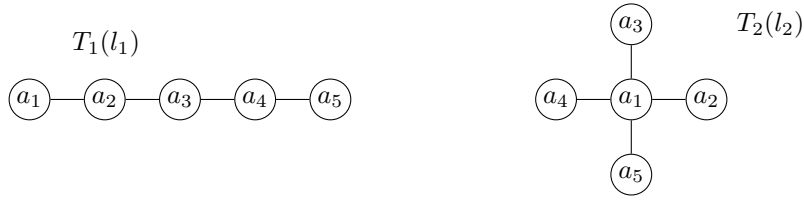
\begin{figure}[h]
\begin{center}
\begin{tikzpicture}[grow=right,
level 1/.style={level distance=1cm,sibling distance=25mm},
level 2/.style={level distance=1cm,sibling distance=12mm},
level 3/.style={level distance=1cm,sibling distance=6mm},
solid node/.style={circle,draw,inner sep=1.5,fill=black},
hollow node/.style={circle,draw,inner sep=1.5}
]

\node [hollow node, label={[label distance=10pt] above right:{\(T_1(l_1)\)}}] at (0,0) {\(a_1\)}
child{node [hollow node]{\(a_2\)}
	child {node [hollow node]{\(a_3\)}
			child {node [hollow node]{\(a_4\)}
				child {node [hollow node]{\(a_5\)}}
			}
	}
};
\tikzset{
level 1/.style={level distance=1cm,sibling angle=90},
level 2/.style={level distance=1cm,sibling distance=12mm},
level 3/.style={level distance=1cm,sibling distance=6mm},
solid node/.style={circle,draw,inner sep=1.5,fill=black},
hollow node/.style={circle,draw,inner sep=1.5}
}
\node(0) [hollow node] at (8,0) {\(a_1\)}
child [grow=0] {node [hollow node]{\(a_2\)}}
child [grow=90] {node [hollow node, label={[label distance=1cm] right:{\(T_2(l_2)\)}}]{\(a_3\)}}
child [grow=180] {node [hollow node]{\(a_4\)}}
child [grow=-90]{node [hollow node]{\(a_5\)}};
\end{tikzpicture}
\end{center}
\caption{The labeled path \(T_1(l_1)\) and the labeled star \(T_2(l_2)\) generate some isometric ultrametric spaces.}
\label{fig2}
\end{figure}

Now we expand the construction of mapping \(d_l\), defined by formula~\eqref{e3.11} for labeled trees, to the case of arbitrary finite connected, labeled graph.

Let \(G = G(l)\) be a connected, finite graph with a labeling \(l \colon V(G) \to \mathbb{R}^{+}\) and let \(\rho_l \colon V(G) \times V(G) \to \mathbb{R}^{+}\) be a mapping defined as
\begin{equation}\label{e3.12}
\rho_l(u, v) = \begin{cases}
0, & \text{if } u = v\\
\min\limits_{P \in \mathcal{P}_{u,v}}\max\limits_{v^{*} \in V(P)} l(v^{*}), & \text{if } u \neq v,
\end{cases}
\end{equation}
where \(\mathcal{P}_{u,v}\) is the set of all paths joining \(u\) and \(v\) in \(G\).

In the formulation of the next theorem we use the notion of \emph{spanning tree}. Recall that a tree \(T\) is a \emph{spanning tree} for a graph \(G\) if \(T \subseteq G\) and \(V(T) = V(G)\) hold. It is well-known that a graph \(G\) is connected if and only if \(G\) contains a spanning tree (see, for example, \cite[Theorem~4.6]{BM}).

\begin{theorem}\label{t3.9}
For every connected graph \(G = G(l)\) with a labeling \(l \colon V(G) \to \mathbb{R}^{+}\) there is a labeled tree \(T = T(l)\) such that \(T\) is a spanning tree for \(G\) and the equality \(\rho_l = d_l\) holds.
\end{theorem}

\begin{proof}
Let us denote by \(\Delta(G)\) the number \(|E(G)| + 1 -|V(G)|\),
\[
\Delta(G) = |E(G)| + 1 -|V(G)|.
\]
We give a proof by induction on \(\Delta(G)\). 

Since \(G\) is connected, we have \(|V(G)| \geqslant |E(G)| + 1\). The equality \(|V(G)| = |E(G)| + 1\) holds if and only if \(G\) is tree (see, for example, \cite[Corollary~1.53]{Die2005}). Hence, for \(\Delta(G) = 0\), the graph \(G\) is a tree and the theorem is valid. Suppose the theorem is valid if \(\Delta(G) \leqslant N\) for some nonnegative integer \(N\). Let \(\Delta(G) = N +1\) hold. Then \(G\) contains a cycle \(C\). Let \(v_1\) be a vertex of \(C\) such that
\begin{equation}\label{t3.9:e2}
l(v_1) = \max\{l(v) \colon v \in V(C)\}
\end{equation}
and let \(u_1\) be a vertex of \(C\) for which the edge \(e_1 = \{u_1, v_1\}\) belongs to \(E(C)\). Write \(G^{1}\) for the graph \(G \setminus \{e_1\}\) obtained from \(G\) by deleting \(e_1\). Then \(G^{1}\) is connected and we have 
\[
\Delta(G^{1}) \leqslant N-1 \quad \text{and} \quad V(G^{1}) = V(G).
\]
Hence, by induction hypothesis, there is a spanning tree \(T^{1}\) with the same labeling \(l \colon V(T^{1}) \to \mathbb{R}^{+}\) and satisfying the equality \(\rho_{l}^{1} = d_l\) with
\begin{equation}\label{t3.9:e3}
\rho_{l}^{1}(u, v) = \begin{cases}
0, & \text{if } u = v\\
\min\limits_{P \in \mathcal{P}_{u,v}^{1}}\max\limits_{v^{*} \in V(P)} l(v^{*}), & \text{if } u \neq v,
\end{cases}
\end{equation}
where \(\mathcal{P}_{u,v}^{1}\) is the set of all paths joining \(u\) and \(v\) in \(G^{1}\). To complete the proof it suffices to show that \(\rho_{l}^{1} = \rho_{l}\). Since \(\mathcal{P}_{u,v}^{1} \subseteq \mathcal{P}_{u,v}\) holds, we have \(\rho_{l}^{1}(u, v) \geqslant \rho_{l}(u, v)\) for all \(u\), \(v \in V(G)\). Hence, \(\rho_{l}^{1} = \rho_{l}\) if and only if
\begin{equation}\label{t3.9:e4}
\rho_{l}^{1}(u, v) \leqslant \rho_{l}(u, v)
\end{equation}
holds for all \(u\), \(v \in V(G)\). Inequality \eqref{t3.9:e4} holds if, for every path \(P \in \mathcal{P}_{u,v}\), there is a path \(P^{1} \in \mathcal{P}_{u,v}^{1}\) such that
\begin{equation}\label{t3.9:e5}
\max\limits_{v^{*} \in V(P)} l(v^{*}) \geqslant \max\limits_{v^{*} \in V(P^{1})} l(v^{*}).
\end{equation}
The existence of the path \(P^{1} \in \mathcal{P}_{u,v}^{1}\) satisfying \eqref{t3.9:e5} is trivial if \(e_1 \notin E(P)\). Let \(e_1 \in E(P)\) and let \(C \setminus e_1\) and \(P \setminus e_1\) be connected graphs obtained from \(C\) and, respectively, from \(P\) by deleting \(e_1\). Write \(W\) for the graph with 
\[
V(W) = V(P) \cup V(C)
\]
and
\[
E(W) = E(P \setminus e_1) \cup E(C \setminus e_1).
\]
Then we have
\begin{multline}\label{t3.9:e6}
\max\limits_{v^{*} \in V(W)} l(v^{*}) = \max\left\{\max\limits_{v^{*} \in V(P \setminus e_1)} l(v^{*}), \max\limits_{v^{*} \in V(C\setminus e_1)} l(v^{*})\right\}\\
= \max\left\{\max\limits_{v^{*} \in V(P \setminus e_1)} l(v^{*}), l(v_1)\right\} = \max\limits_{v^{*} \in V(P \setminus e_1)} l(v^{*}) = \max\limits_{v^{*} \in V(P)} l(v^{*}).
\end{multline}
Since \(W\) is a connected graph, there is a path \(P^{2} \subseteq W\) joining \(u\) and \(v\) in \(W\). Consequently,
\begin{equation}\label{t3.9:e7}
\max\limits_{v^{*} \in V(P^{2})} l(v^{*}) \leqslant \max\limits_{v^{*} \in V(W)} l(v^{*})
\end{equation}
holds. Inequality \eqref{t3.9:e5} follows from \eqref{t3.9:e6} and \eqref{t3.9:e7} with \(P^{1} = P^{2}\).
\end{proof}

\begin{corollary}\label{c3.10}
The following statements are equivalent for every connected, labeled graph \(G = G(l)\):
\begin{enumerate}
\item\label{c3.10:s1} The function \(\rho_l\) is an ultrametric on \(V(G)\).
\item\label{c3.10:s2} The function \(\rho_l\) is a metric on \(V(G)\).
\item\label{c3.10:s3} The inequality
\begin{equation}\label{c3.10:e1}
\max\{l(u_1), l(v_1)\} > 0
\end{equation}
holds for every \(\{u_1, v_1\} \in E(G)\).
\end{enumerate}
\end{corollary}

\begin{proof}
Proposition~\ref{p3.5} and Theorem~\ref{t3.9} imply that \(\ref{c3.10:s1} \Leftrightarrow \ref{c3.10:s2}\) is valid. 

Let \(G = G(l)\) be a connected, labeled graph and let \(T = T(l)\) be a spanning tree such that \(d_l = \rho_l\) holds. Since \(E(G) \supseteq E(T)\) holds, statement \ref{c3.10:s3} of the present corollary implies statement \ref{p3.5:s3} of Proposition~\ref{p3.5}. Thus, we also have \(\ref{c3.10:s3} \Rightarrow \ref{c3.10:s1}\).

Let us prove the validity \(\ref{c3.10:s1} \Rightarrow \ref{c3.10:s3}\). Let \ref{c3.10:s1} hold. Suppose now that \(l(u_1) = l(v_1) = 0\) for some \(\{u_1, v_1\} \in E(G)\). Then using~\eqref{e3.12} we obtain \(\rho_{l}(u_1, v_1) = 0\), contrary to \ref{c3.10:s1}.
\end{proof}

Example~\ref{ex3.7} gives rise the following problem (cf.~Problem~\ref{pr1}).

\begin{problem}\label{pr2}
Let \(G_1 = G_1(l_1)\) and \(G_2 = G_2(l_2)\) be labeled, connected graphs such that 
\[
\max\{l_i(u_i), l_i(v_i)\} > 0
\]
for every \(\{u_i, v_i\} \in E(G_i)\), \(i = 1\), \(2\). Find conditions under which
\begin{equation}\label{pr2:e1}
(G_1(l_1) \simeq G_2(l_2)) \Leftrightarrow ((V(G_1), \rho_{l_1}) \simeq (V(G_2), \rho_{l_2}))
\end{equation}
is valid. Is the statement 
\begin{itemize}
\item Every isometry \(f \colon V(G_1) \to V(G_2)\) of the ultrametric spaces \((V(G_1), \rho_{l_1})\) and \((V(G_2), \rho_{l_2})\) is an isomorphism of \(G_1(l_1)\) and \(G_2(l_2)\)
\end{itemize}
true if \eqref{pr2:e1} is valid?
\end{problem}

In what follows we give a partial solution of Problem~\ref{pr2} for labeled rooted trees.

\begin{theorem}\label{t3.12}
Let \(T_i = T_i(r_i, l_i)\) be labeled rooted tree such that \(\delta^{+}(u_i) \neq 1\) and
\begin{equation}\label{t3.12:f1}
(\delta^+(u_i) =0) \Leftrightarrow (l(u_i)=0)
\end{equation}
for every \(u_i \in V(T_i)\) and, in addition, \(l(v_i) < l(u_i)\) holds whenever $v_i$ is a direct successor of $u_i$, \(i = 1\), \(2\). Then the following statements hold:
\begin{enumerate}
\item \label{t3.12:s1} \((T_1(l_1) \simeq T_2(l_2)) \Leftrightarrow ((V(T_1), d_{l_1}) \simeq (V(T_2), d_{l_2}))\) is valid.
\item \label{t3.12:s2} If the ultrametric spaces \((V(T_1), d_{l_1})\) and \((V(T_2), d_{l_2})\) are isometric, then the following conditions are equivalent:
\begin{enumerate}
\item \label{t3.12:s2:s1} \(|V(T_1)| = |V(T_2)| = 1\).
\item \label{t3.12:s2:s2} Every isometry \(f \colon V(T_1) \to V(T_2)\) of \((V(T_1), d_{l_1})\) and \((V(T_2), d_{l_2})\) is an isomorphism of \(T_1(l_1)\) and \(T_2(l_2)\).
\end{enumerate}
\end{enumerate}
\end{theorem}

In proving this theorem we will use some results describing the structure of balleans \(\mathbf{B}_X\) of finite ultrametric spaces \((X, d)\) (see, in particular, Proposition~\ref{p2.8}).

Let \(A\) and \(B\) be two nonempty bounded subsets of a metric space \((X, d)\). The \emph{Hausdorff distance} \(d_H(A, B)\) between \(A\) and \(B\) is defined by
\begin{equation}\label{e3.28}
d_H(A, B) := \max\{\sup_{a \in A} \inf_{b \in B} d(a, b), \sup_{b \in B} \inf_{a \in A} d(a, b)\}.
\end{equation}
The definition and some properties of the Hausdorff distance can be found in~\cite{BBI2001}. See also \cite{Qiu2009pNUAA, Qiu2014pNUAA} for properties of Hausdorff distance in ultrametric spaces. We note only that if \(\{a\}\) and \(\{b\}\) are singular balls in \((X, d)\), then~\eqref{e3.28} implies
\begin{equation*}
d_H(\{a\}, \{b\}) = d(a, b).
\end{equation*}

The following lemma is a part of Theorem~2.5 from~\cite{Dov2019pNUAA}.

\begin{lemma}\label{l3.13}
Let $(X,d)$ be a finite nonempty ultrametric space with the representing tree $T_X$ and let $B_1$ and $B_2$ be distinct balls in $(X, d)$. If \(P\) is the path joining \(B_1\) and \(B_2\) in \(T_X\), then
\begin{equation*}
d_H(B_1, B_2) = \max_{u \in V(P)} l(u).
\end{equation*}
\end{lemma}

Lemma~\ref{l3.13} and Proposition~\ref{p3.5} imply that \((\mathbf{B}_{X}, d_H)\) is a finite ultrametric space for every finite ultrametric space \((X, d)\). Now we want to describe the structure of the representing tree \(T_{\mathbf{B}_{X}}\).

\begin{definition}\label{d3.14}
Let \(T_1\) and \(T_2\) be trees and let \(x\) be a leaf of \(T_2\). Suppose we have
\[
V(T_2) = V(T_1) \cup \{x\}, \quad x \notin V(T_1), \quad E(T_1) \subseteq E(T_2).
\]
Then there is a unique vertex \(u \in V(T_1)\), such that \(\{x, u\} \in E(T_2)\). In this case we say that \(T_2\) is obtained by \emph{adding the leaf} \(x\) \emph{to the vertex} \(u\).
\end{definition}

In the following lemma we consider an ultrametric space \((X, d)\) with \(X\) satisfying the condition
\begin{equation}\label{eq3.5}
\{Y\} \not\subseteq X
\end{equation}
for every \(Y \subseteq X\). We note that for every ultrametric space \((Z, \rho)\) there is an ultrametric space \((X, d)\) such that \((X, d) \simeq (Z, \rho)\) and \eqref{eq3.5} holds for every \(Y \subseteq X\).

\begin{lemma}\label{l3.15}
Let \((X, d)\) be a finite ultrametric space with the representing tree \(T_X = T_X(X, l)\), let \(\mathbf{B}_{X}\) be the ballean of \((X, d)\) and let \(d_H\) be the Hausdorff distance on \(\mathbf{B}_{X}\). Then the representing tree \(T_{\mathbf{B}_{X}}\) of the ultrametric space \((\mathbf{B}_{X}, d_H)\) and the labeled rooted tree \(T_1 = T_1(r_1, l_1)\), \(r_1 = X\), obtained from \(T_X\) by adding a leaf to every internal vertex of \(T_X\) and labeled such that 
\[
l_1(v_1) = \begin{cases}
l(v_1), & \text{if } v_1 \in V(T_X)\\
0, & \text{if } v_1 \notin V(T_X),
\end{cases}
\]
are isomorphic as labeled rooted trees.
\end{lemma}

For the proof of this lemma see Theorem~3.10 in \cite{Dov2019pNUAA}.

Our next lemma is a direct consequence of Lemma~\ref{l3.15}.

\begin{lemma}\label{l3.16}
Let \((X_1, d_1)\) and \((X_2, d_2)\) be finite, nonempty ultrametric spaces. Then the equivalence
\[
(T_{X_1} \simeq T_{X_2}) \Leftrightarrow (T_{\mathbf{B}_{X_1}} \simeq T_{\mathbf{B}_{X_2}})
\]
is valid.
\end{lemma}

In the proof of Theorem~\ref{t3.12} we also use the next simple fact.

\begin{lemma}\label{l3.17}
If \(f \colon V(T_1) \to V(T_2)\) is an isomorphism of labeled trees \(T_1 = T_1(l_1)\) and \(T_2 = T_2(l_2)\), then the equality
\[
d_{l_1}(u_1, v_1) = d_{l_2} (f(u_1), f(v_1))
\]
holds for all \(u_1\), \(v_1 \in V(T_1)\).
\end{lemma}

\begin{proof}
It follows from Definition~\ref{d1.6} and \eqref{e3.11}.
\end{proof}

\begin{proof}[Proof of Theorem~\ref{t3.12}]
\ref{t3.12:s1}. By Corollary~\ref{c3.6} the implication
\[
(T_1(l_1) \simeq T_2(l_2)) \Rightarrow ((V(T_1), d_{l_1}) \simeq (V(T_2), d_{l_2}))
\]
is valid. Suppose that \((V(T_1), d_{l_1})\) and \((V(T_2), d_{l_2})\) are isometric. Since \((V(T_i), d_{l_i})\) is ultrametric, Theorem~\ref{t2.7} implies that 
\[
((V(T_1), d_{l_1}) \simeq (V(T_2), d_{l_2})) \Leftrightarrow (T_{V(T_1)} \simeq T_{V(T_2)}),
\]
where \(T_{V(T_1)}\) and \(T_{V(T_2)}\) are the representing trees of \((V(T_1), d_{l_1})\) and \((V(T_2), d_{l_2})\), respectively. Thus, 
\begin{equation}\label{t3.12:e0}
T_{V(T_1)} \simeq T_{V(T_2)}
\end{equation}
is valid. Using Theorem~\ref{t1.7} we find finite ultrametric spaces \((X_1, d_1)\) and \((X_2, d_2)\) such that
\begin{equation}\label{t3.12:e1}
T_{X_1} \simeq T(r_1, l_1) \quad \text{and} \quad T_{X_2} \simeq T(r_2, l_2).
\end{equation}
Proposition~\ref{p2.8}, Lemma~\ref{l3.13} and \eqref{t3.12:e1} imply that
\begin{equation}\label{t3.12:e2}
(\mathbf{B}_{X_i}, d_{H_i}) \simeq (V(T_i), d_{l_i}), \quad i =1,2,
\end{equation}
where \(d_{H_1}\) and \(d_{H_2}\) are the Hausdorff distances generated by \(d_1\) and \(d_2\), respectively (see~\eqref{e3.28}). By Theorem~\ref{t2.7}, statement \eqref{t3.12:e2} implies
\begin{equation}\label{t3.12:e3}
\mathbf{B}_{X_i} \simeq T_{V(T_i)}.
\end{equation}
Now using \eqref{t3.12:e0} and \eqref{t3.12:e3} we have \(\mathbf{B}_{X_1} \simeq \mathbf{B}_{X_2}\). By Lemma~\ref{l3.16},
\begin{equation}\label{t3.12:e4}
T_{X_1} \simeq T_{X_2}
\end{equation}
is valid. From \eqref{t3.12:e1} and \eqref{t3.12:e4} it follows that 
\[
T(r_1, l_1) \simeq T(r_2, l_2),
\]
which implies \(T(l_1) \simeq T(l_2)\). Statement \ref{t3.12:s1} is proved.

\ref{t3.12:s2}. Let \((V(T_1), d_{l_1})\) and \((V(T_2), d_{l_2})\) be isometric. The implication \(\ref{t3.12:s2:s1} \Rightarrow \ref{t3.12:s2:s2}\) is trivially valid. Suppose \ref{t3.12:s2:s2} holds, but \ref{t3.12:s2:s1} is false. Then we have 
\begin{equation}\label{t3.12:e5}
|V(T_1)| = |V(T_2)| \geqslant 2.
\end{equation}
By condition of the theorem, the attitude
\begin{equation}\label{t3.12:e6}
\delta^{+}(r_1) \neq 1 \neq \delta^{+}(r_2)
\end{equation}
satisfied. Using \eqref{t3.12:e5} and \eqref{t3.12:e6} we obtain 
\begin{equation}\label{t3.12:e7}
\delta^{+}(r_1) \geqslant 2 \quad \text{and} \quad \delta^{+}(r_2) \geqslant 2.
\end{equation}
By statement \ref{t3.12:s1}, from \((V(T_1), d_{l_1}) \simeq (V(T_2), d_{l_2})\) it follows that there is an isomorphism \(f \colon V(T_1) \to V(T_2)\) of \(T_1(l_1)\) and \(T_2(l_2)\). Lemma~\ref{l3.17} implies that the mapping \(f\) is also an isometry of \((V(T_1), d_{l_1})\) and \((V(T_2), d_{l_2})\).

Let \(v_1^{*}\) be a leaf of \(T_1\) and let \(u_1^{*}\) be the unique vertex of \(T_1\) such that \(\{v_1^{*}, u_1^{*}\} \in E(T_1)\). Since \(f\) is an isomorphism of \(T_1(l_1)\) and \(T_2(l_2)\), the vertex \(v_2^{*} = f(v_1^{*})\) is a leaf of \(T_2\) and \(\{v_2^{*}, u_2^{*}\} \in E(T_2)\) holds with \(u_2^{*} = f(u_1^{*})\). Using \eqref{t3.12:f1} we also see that
\[
l_1(v_1^{*}) = l_2(v_2^{*}) = 0 \quad \text{and} \quad l_1(u_1^{*}) = l_2(u_2^{*}) > 0.
\]
Let us define the bijection \(g \colon V(T_2) \to V(T_2)\) as
\begin{equation}\label{t3.12:e9}
g(v_2) = \begin{cases}
v_2^{*}, & \text{if } v_2 = u_2^{*}\\
u_2^{*}, & \text{if } v_2 = v_2^{*}\\
v_2, & \text{otherwise}.
\end{cases}
\end{equation}
Then the composition
\begin{equation}\label{t3.12:e8}
V(T_1) \xrightarrow{f} V(T_2) \xrightarrow{g} V(T_2)
\end{equation}
is not an isomorphism \(T_1\) and \(T_2\) because \(v_1^{*}\) is a leaf of \(T_1\), but \(g(f(v_1^{*}))\) is an inner node of \(T_2\), \(g(f(v_1^{*})) = g(v_2^{*}) = u_2^{*}\). 

Since \(f\) is an isometry, mapping \eqref{t3.12:e8} is an isometry if and only if \(g\) is an isometry. Using \eqref{t3.12:e9} and Definition~\ref{d1.1} we can simply prove that \(g\) is an isometry if and only if
\begin{equation}\label{t3.12:e10}
d_{l_2} (x, u_2^{*}) = d_{l_2} (x, v_2^{*})
\end{equation}
holds whenever 
\begin{equation}\label{t3.12:e11}
u_2^{*} \neq x \neq v_2^{*}
\end{equation}
and \(x \in V(T_2)\).

Let \(x \in V(T_2)\) satisfy \eqref{t3.12:e11}. Let us consider the path \(P_1 = (v_1, \ldots, v_n)\) joining \(v_2^{*} = v_1\) and \(x = v_n\), and the path \(P_2 = (u_1, \ldots, u_m)\) joining \(u_2^{*} = u_1\) and \(x = u_m\). Since \(v_2^{*}\) is a leaf of \(T_2\), \(\{u_2^{*}, v_2^{*}\} \in E(T_2)\) (because \(\{u_1^{*}, v_1^{*}\} \in E(T_1)\) and \(f\) is an isomorphism of \(T_1\) and \(T_2\)), we obtain the equalities 
\[
m+1 = n \quad \text{and} \quad v_2 = u_1, \ v_3 = u_2, \ \ldots, \ v_n = u_{m-1} = u_m.
\]
Now equality~\eqref{t3.12:e10} follows from~\eqref{e3.11} and the equality \(l_2(v_2^{*}) = 0\).
\end{proof}

We conclude this section with a simple example of connected labeled graph \(G(l)\) for which the group of isomorphisms of \(G(l)\) coincides with the group isometries of \((V(G), \rho_{l})\).

\begin{example}
Let \(G = G(l)\) be the labeled graph shown in Figure~\ref{fig6}. Then every isometry \(f \colon V(G) \to V(G)\) from \((V(G), \rho_{l})\) to \((V(G), \rho_{l})\) is an isomorphism from \(G(l)\) to \(G(l)\).
\end{example}

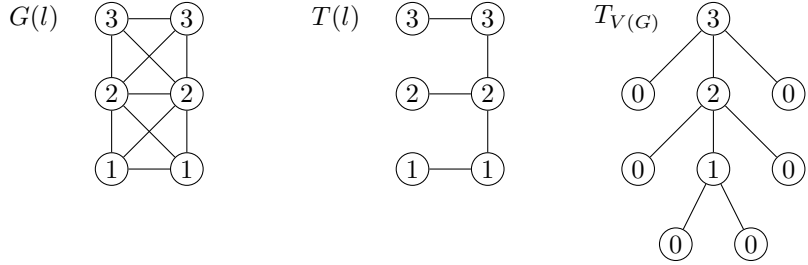
\begin{figure}[h]
\begin{center}
\begin{tikzpicture}[
node distance=1cm, on grid=true,
level 1/.style={level distance=1cm,sibling distance=1cm},
level 2/.style={level distance=1cm,sibling distance=1cm},
level 3/.style={level distance=1cm,sibling distance=1cm},
solid node/.style={circle,draw,inner sep=1.5,fill=black},
hollow node/.style={circle,draw,inner sep=1.5}
]

\node [hollow node, label={[label distance=10pt] left:{\(G(l)\)}}] (A) at (0,0) {\(3\)};
\node [hollow node] (B) [right=of A] {\(3\)};
\node [hollow node] (C) [below=of A] {\(2\)};
\node [hollow node] (D) [below=of B] {\(2\)};
\node [hollow node] (E) [below=of C] {\(1\)};
\node [hollow node] (F) [below=of D] {\(1\)};
\draw (A) -- (B) -- (D) -- (F) -- (E) -- (C) -- (A) -- (D) -- (E);
\draw (B) -- (C) -- (F);
\draw (C) -- (D);

\node [hollow node, label={[label distance=10pt] left:{\(T(l)\)}}] (A1) at (4,0) {\(3\)};
\node [hollow node] (B1) [right=of A1] {\(3\)};
\node [hollow node] (C1) [below=of A1] {\(2\)};
\node [hollow node] (D1) [below=of B1] {\(2\)};
\node [hollow node] (E1) [below=of C1] {\(1\)};
\node [hollow node] (F1) [below=of D1] {\(1\)};
\draw (A1) -- (B1) -- (D1) -- (F1) -- (E1);
\draw (C1) -- (D1);

\node [hollow node, label={[label distance=10pt] left:{\(T_{V(G)}\)}}] (A2) at (8,0) {\(3\)}
	child {node [hollow node]{\(0\)}}
	child {node [hollow node]{\(2\)}
		child {node [hollow node]{\(0\)}}
		child {node [hollow node]{\(1\)}
			child {node [hollow node]{\(0\)}}
			child {node [hollow node]{\(0\)}}
		}
		child {node [hollow node]{\(0\)}}
	}
	child {node [hollow node]{\(0\)}};
\end{tikzpicture}
\end{center}
\caption{Here \(T = T(l)\) is a spanning tree for \(G(l)\) satisfying the equality \(\rho_{l} = d_{l}\) and \(T_{V(G)}\) is the representing tree for the ultrametric space \((V(G), \rho_{l})\).}
\label{fig6}
\end{figure}

\section{Isomorphisms of monotone trees and of equidistant trees}\label{sec4}

Let $T = T(r)$ be a rooted tree. Write
\begin{equation}\label{e3.1}
V_0^{+}(T) := \{v \in V(T) \colon \delta^{+}(v) = 0\}.
\end{equation}
It is clear that \(V_0^{+}(T) \subseteq L(T)\) and \(V_0^{+}(T) = L(T)\) if and only if \(r \notin L(T)\) or \(V(T) = \{r\}\).

\begin{definition}\label{d3.1}
Let \(T = T(r, w)\) be a weighted rooted tree with strictly positive weight. The weight $w$ is \emph{equidistant} if there is a constant $K$ such that, for every $u \in V_{0}^{+}(T)$,
\begin{equation}\label{e6.1}
K = \sum_{i=1}^{n-1} w(\{v_i, v_{i+1}\}),
\end{equation}
where $(v_1, \ldots, v_n)$ is the path joining the root $r = v_1$ with the leaf $u = v_n$ in $T$. In this case we say that \(T(r, w)\) is \emph{equidistant}.
\end{definition}

Note that the root $r$ of an equidistant tree $T = T(r, w)$ can be a leaf of $T$ (see Figure~\ref{fig10}).

We will also say that a labeling $l \colon V(T) \to \mathbb{R}^{+}$ of a labeled rooted tree $T = T(r, l)$ is \emph{monotone} if
$$
l^{-1}(0) = V_0^{+}(T)
$$
and, in addition, the inequality
\begin{equation}\label{e3.3}
l(v) < l(u)
\end{equation}
holds whenever $v$ is a direct successor of $u$. 

A tree is \emph{monotone} if it is a labeled rooted tree with monotone labeling. By Theorem~\ref{t1.7}, the tree \(T_X\) is monotone for every finite ultrametric space \((X, d)\).

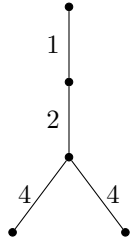
\begin{figure}[h]
\begin{tikzpicture}[scale=1]
\tikzstyle{level 1}=[level distance=10mm,sibling distance=10mm]
\tikzstyle{level 2}=[level distance=10mm,sibling distance=10mm]
\tikzstyle{level 3}=[level distance=10mm,sibling distance=15mm]
\tikzset{
solid node/.style={circle,draw,inner sep=1pt,fill=black},
hollow node/.style={circle,draw,inner sep=1.5}
}

\node [solid node] at (0,0) {}
	child {node [solid node]{}
		child {node [solid node]{}
			child {node [solid node]{} edge from parent node[left]{$4$}}
			child {node [solid node]{} edge from parent node[right]{$4$}}
			edge from parent node[left]{$2$}
		}
		edge from parent node[left]{$1$}
	};
\end{tikzpicture}
\caption{The rooted trees in which the root has degree $1$ are known as planted trees. It is an example of planted, equidistant tree satisfying equality~\eqref{e6.1} with $K = 7$.}
\label{fig10}
\end{figure}

The following proposition gives us a tool for investigation of a duality between the equidistant trees and monotone trees.

\begin{proposition}\label{p6.1}
Let \(T = T(r)\) be a nonempty rooted tree. Then the following statements hold:
\begin{enumerate}
\item\label{p6.1:s1} For every monotone labeling \(l \colon V(T) \to \mathbb{R}^+\) there is the unique equidistant weight \(w \colon E(T) \to \mathbb{R}^+\) such that
\begin{equation}\label{p6.1e1}
w(\{u, v\}) = \frac{1}{2}(l(u) - l(v))
\end{equation}
holds whenever \(v\) is a direct successor of \(u\).
\item\label{p6.1:s2} For every equidistant weight \(w \colon E(T) \to \mathbb{R}^+\) there is the unique monotone labeling \(l \colon V(T) \to \mathbb{R}^+\) such that \eqref{p6.1e1} holds whenever \(v\) is a direct successor of \(u\).
\end{enumerate}
\end{proposition}

\begin{proof}
\ref{p6.1:s1} Let $l \colon V(T) \to \mathbb{R}^+$ be monotone. Since for every \(\{u, v\} \in E(T)\) either \(u\) is a direct successor of \(v\) or \(v\) is a direct successor of \(u\), there is the unique weight $w \colon E(T) \to \mathbb{R}^+$ for which~\eqref{p6.1e1} holds whenever $v$ is a direct successor of $u$. We must prove that $w$ is equidistant.

Equality~\eqref{p6.1e1} implies that $w(e) > 0$ for every $e \in E(T)$, i.e., \(w\) is strictly positive. Let $v \in V_0^{+}(T)$ and let $(v_1, \ldots, v_n)$ be a path in $T$ such that $v_1 = r$ and $v_n = v$. Then using~\eqref{p6.1e1} and the equality $l(v) = 0$ we obtain
\begin{multline*}
\sum_{i=1}^{n-1} w(\{v_i, v_{i+1}\}) = \sum_{i=1}^{n-1} \frac{1}{2} (l(v_i) - l(v_{i+1})) \\
= \frac{1}{2} (l(v_1) - l(v_n)) = \frac{1}{2} (l(r) - l(v)) = \frac{1}{2} l(r).
\end{multline*}
Hence~\eqref{e6.1} holds with $K = \frac{1}{2}l(r)$. Thus $w$ is equidistant.

\ref{p6.1:s2} Let $w \colon E(T) \to \mathbb{R}^{+}$ be equidistant. For every \(u \in V(T)\) we define \(l(u)\) by
\begin{equation}\label{p6.1:e2}
2l(u) = \begin{cases}
K, & \text{if $u = r$}\\
K - \sum_{j=1}^{n-1} w(\{u_j, u_{j+1}\}), & \text{if $u \neq r$,}
\end{cases}
\end{equation}
where $u_1 = r$, $u_n = u$ and $(u_1, \ldots, u_n)$ is the unique path joining the root $r$ and the node $u$ in $T$ and $K$ is the constant defined by~\eqref{e6.1}. Then \eqref{e6.1} and \eqref{p6.1:e2} imply that \(l(v_0) = 0\) holds for every \(v_0 \in V_0^{+}(T)\). Moreover, if \(v\) is a direct successor of \(u\), then using \eqref{p6.1:e2} we obtain
\[
2l(v) = 2l(u) - w(\{u, v\}),
\]
which implies \eqref{p6.1e1}.

Suppose that \(l_1 \colon V(T) \to \mathbb{R}^{+}\) is a monotone labeling such that \(l_1 \neq l\) but
\begin{equation}\label{p6.1:e3}
w(\{u, v\}) = \frac{1}{2} (l_1(u) - l_1(v))
\end{equation}
holds whenever \(v\) is a direct successor of \(u\). Since \(l_1\) and \(l\) are monotone, we have \(l_1(v) = l(v)\) for every \(v \in V_0^{+}(T)\). Hence, there is 
\begin{equation}\label{p6.1:e4}
u^{*} \in V(T) \setminus V_0^{+}(T)
\end{equation}
such that \(l_1(u^{*}) \neq l(u^{*})\) but \(l_1(v) = l(v)\) for every successor of \(u^{*}\). Condition~\eqref{p6.1:e4} implies that the set of all successors of \(u^{*}\) is nonempty. Let \(v^{*}\) be a direct successor of \(u^{*}\). Then using \eqref{p6.1e1} and \eqref{p6.1:e3} we obtain
\begin{align*}
\frac{1}{2} (l_1(u^{*}) - l_1(v^{*})) &= w(\{u^{*}, v^{*}\}) = \frac{1}{2} (l(u^{*}) - l(v^{*})),\\
l_1(u^{*}) - l_1(v^{*}) &=l(u^{*}) - l(v^{*}),\\
l_1(u^{*}) &=l(u^{*}),
\end{align*}
contrary to \(l_1(u^{*}) \neq l(u^{*})\). Statement \ref{p6.1:s2} is proved.
\end{proof}

Figure~\ref{fig11} gives us an example of equidistant tree and the corresponding monotone tree.

\begin{figure}[h]
\begin{center}
\begin{tikzpicture}
\tikzstyle{level 1}=[level distance=15mm,sibling distance=25mm]
\tikzstyle{level 2}=[level distance=15mm,sibling distance=12mm]
\tikzstyle{level 3}=[level distance=15mm,sibling distance=6mm]
\tikzset{
solid node/.style={circle,draw,inner sep=1.5,fill=black},
hollow node/.style={circle,draw,inner sep=1.5}
}
\node (0) [solid node, label=left:{\(T(r,w)\)}] at (0,0) {}
child [sibling distance=15mm] {node(1) [solid node]{} edge from parent node[left]{$5$}}
child{node(1) [solid node]{}
	child{node[solid node]{}
		child{node[solid node]{} edge from parent node[left]{$2$}}
		child{node[solid node]{} edge from parent node[right]{$2$}}
		edge from parent node[left]{$2$}
		}
	child{node[solid node]{}
		child{node[solid node]{} edge from parent node[left]{$2$}}
		child{node[solid node]{} edge from parent node[right]{$2$}}
		edge from parent node[left]{$2$}
		}
	child [sibling distance=6mm] {node[solid node]{} edge from parent node[left]{$4$}}
	child [sibling distance=6mm] {node[solid node]{} edge from parent node[right]{$4$}}
	edge from parent node[left,xshift=3]{$1$}
}
child{node[solid node]{}
child{node[solid node]{} edge from parent node[left]{$3$}}
child{node[solid node]{} edge from parent node[right]{$3$}}
edge from parent node[right,xshift=3]{$2$}
};

\tikzset{
solid node/.style={circle,draw,inner sep=1.5},
hollow node/.style={circle,draw,inner sep=1.5}
}
\node(0)[solid node, label=left:{\(T(r,l)\)}] at (6,0) {$10$}
child [sibling distance=15mm] {node(1) [solid node]{$0$}}
child {node(1) [solid node]{$8$}
	child{node[solid node]{$4$}
		child{node[solid node]{$0$}}
		child{node[solid node]{$0$}}
	}
	child{node[solid node]{$4$}
		child{node[solid node]{$0$}}
		child{node[solid node]{$0$}}
	}
	child [sibling distance=6mm] {node[solid node]{$0$}}
	child [sibling distance=6mm] {node[solid node]{$0$}}
}
child{node[solid node]{$6$}
	child{node[solid node]{$0$}}
	child{node[solid node]{$0$}}
};
\end{tikzpicture}
\end{center}
\caption{The weight \(w\) and labeling \(l\) satisfy \eqref{p6.1e1} whenever \(v\) is a direct successor of \(u\).}
\label{fig11}
\end{figure}
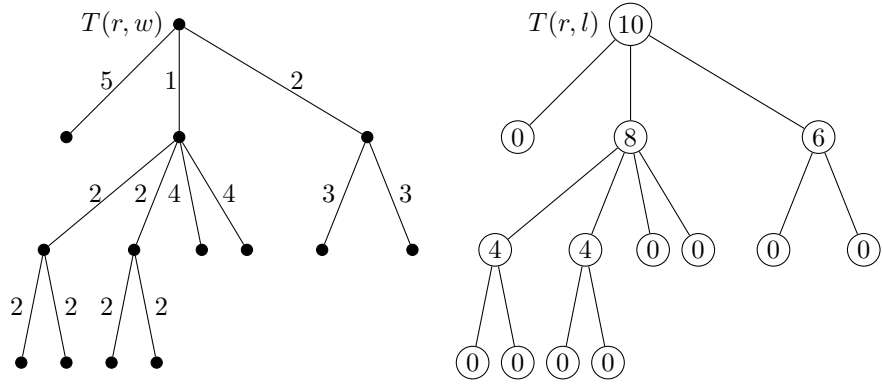

In what follows we write $\hat{l}*w$ and \(\hat{w}*l\) for the monotone labeling and, respectively, for equidistant weight obtained from the equidistant $w \colon E(T) \to \mathbb{R}^{+}$ and, respectively, from the monotone $l \colon V(T) \to \mathbb{R}^{+}$ by~\eqref{p6.1e1}.

\begin{proposition}\label{p3.3}
Let $T_1 = T_1(r_1)$ and \(T_2 = T_2(r_2)\) be nonempty rooted trees. Then the following statements hold:
\begin{enumerate}
\item\label{p3.3:s1} $(T_1(r_1, \hat{l}*w_1) \simeq T_2(r_2, \hat{l}*w_2)) \Leftrightarrow (T_1(r_1, w_1) \simeq T_2(r_2, w_2))$\\
is valid for all equidistant weights $w_1 \colon E(T_1) \to \mathbb{R}^{+}$ and $w_2 \colon E(T_2) \to \mathbb{R}^{+}$.
\item\label{p3.3:s2} $(T_1(r_1, \hat{w}*l_1) \simeq T_2(r_2, \hat{w}*l_2)) \Leftrightarrow (T_1(r_1, l_1) \simeq T_2(r_2, l_2))$\\
is valid for all monotone labelings $l_1 \colon V(T_1) \to \mathbb{R}^{+}$ and $l_2 \colon V(T_2) \to \mathbb{R}^{+}$.
\item\label{p3.3:s3} The equalities
\begin{equation}\label{p3.3:e1}
\hat{w}*(\hat{l}*w_1) = w_1 \quad \text{and}\quad \hat{l}*(\hat{w}*l_1) = l_1
\end{equation}
are satisfied for every equidistant weight $w_{1} \colon E(T_{1}) \to \mathbb{R}^{+}$ and every monotone labeling $l_{1} \colon V(T_{1}) \to \mathbb{R}^{+}$.
\end{enumerate}
\end{proposition}

\begin{proof}
Statements \ref{p3.3:s1} and \ref{p3.3:s2} follows directly Proposition~\ref{p6.1} and the definitions of isomorphic weighted graphs and of isomorphic labeled graphs.

\ref{p3.3:s3} Let \(w_1 \colon E(T_1) \to \mathbb{R}^{+}\) be an equidistant weight. By statement \ref{p6.1:s1} of Proposition~\ref{p6.1}, \(l = \hat{l}*w_1\) is the unique monotone labeling satisfying
\begin{equation}\label{p3.3:e2}
\frac{1}{2} (\hat{l}*w_1(u) - \hat{l}*w_1(v)) = w_1(\{u, w\})
\end{equation}
whenever \(v\) is a direct successor of \(u\). By statement \ref{p6.1:s2} of Proposition~\ref{p6.1}, \(\hat{w}*(\hat{l}* w_1)\) is the unique equidistant weight satisfying
\begin{equation}\label{p3.3:e3}
\frac{1}{2} (\hat{l}*w_1(u) - \hat{l}*w_1(v)) = \hat{w}*(\hat{l}* w_1)(\{u, w\})
\end{equation}
whenever \(v\) is a direct successor of \(u\). Equalities \eqref{p3.3:e2} and \eqref{p3.3:e3} imply the first equality in~\eqref{p3.3:e1}. The second one can be proved similarly.
\end{proof}

The constant \(K\) in Definition~\ref{d3.1} of equidistant trees has a simple geometric interpretation. It is the distance between the root \(r\) and arbitrary \(v_0 \in V_0^{+}(T)\) in the metric space \((V(T), d_{w})\), where \(d_{w}\) is the additive metric generated by equidistant weight \(w\). 

Analogously, if \(l = \hat{l} * w\), then, for every \(v \in V(T)\), the value \(\frac{1}{2} l(v)\) is the distance in \((V(T), d_{w})\) between \(v\) and arbitrary \(v_0 \in V_0^{+}(T_v)\), where \(T_v\) is the rooted subtree of \(T(r, w)\) lying below \(v\) (see~\eqref{e2.5}, \eqref{e2.6}).

\begin{remark}\label{r4.5}
Let \(C\) be an arbitrary strictly positive real number. Proposition~\ref{p6.1} remains valid if we replace formula \eqref{p6.1e1} with the formula 
\[
w(\{u, w\}) = C(l(u) - l(v)),
\]
but the following two lemmas are valid only in the case \(C = \frac{1}{2}\).
\end{remark}

\begin{figure}[ht]
\begin{center}
\begin{tikzpicture}[%
level 1/.style={level distance=1cm,sibling distance=4cm},
level 2/.style={level distance=1cm,sibling distance=3cm},
level 3/.style={level distance=1cm,sibling distance=2.5cm},
level 4/.style={level distance=1cm,sibling distance=1cm},
level 5/.style={level distance=1cm,sibling distance=1cm},
solid node/.style={circle,draw,inner sep=1.5,fill=black},
hollow node/.style={circle,draw,inner sep=1.5},
edge from parent/.style = {draw, postaction={decorate}, decoration={markings,mark=at position .55 with {\arrow[draw, line width=1pt]{<}}}}
]
\node (root1) [solid node, label=above:{\(x_1\)}] at (0,0) {}
	child{node[solid node, label=left:{\(x_2\)}]{}
		child{node[solid node, label=left:{\(x_4\)}]{}
			child{node[solid node, label=left:{\(x_5\)}]{}
				child[level distance=2cm] {node[solid node, label=below:{\(x_8\)}]{}}
				child[level distance=2cm] {node[solid node, label=below:{\(x_9\)}]{}}
			}
			child{node[solid node, label=right:{\(x_6\)}]{}
				child{node[solid node, label=right:{\(x_7\)}]{}
					child{node[solid node, label=below:{\(x_{10}\)}]{}}
					child{node[solid node, label=below:{\(x_{11}\)}]{}}
					child{node[solid node, label=below:{\(x_{12}\)}]{}}
				}
			}
		}
	}
	child [level distance=5cm] {node[solid node, label=below:{\(x_3\)}]{}};
\node (title1) [right=of root1, label=right:{\(T\)}] {};
\end{tikzpicture}
\end{center}
\caption{Let \(x\) be a root of \(T\), \(u_1 = x_8\) and \(u_2 = x_{12}\). Then \(P = (u_1, v_1, v_2, v_3, v_4, u_2) = (x_8, x_5, x_4, x_6, x_7, x_{12})\) and \(v_{i^*} = x_4\). Vertex \(x_4\) is the first common vertex of two paths which join \(x_8\) and \(x_{12}\) with the root \(r = x_1\).}
\label{fig4}
\end{figure}
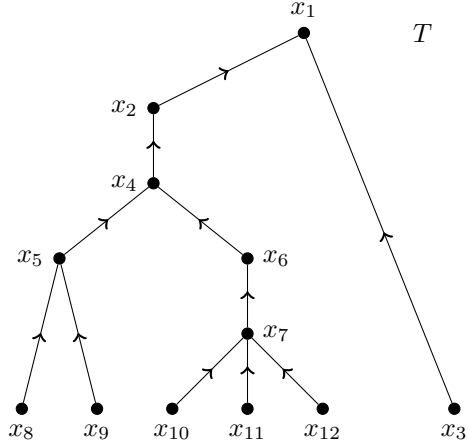

\begin{lemma}\label{l6.4}
Let $T = T(r, w)$ be an equidistant tree, let $l = \hat{l}*w$ be the corresponding monotone labeling, let $u_1$, $u_2$ be two different points of the set $V_{0}^{+}(T)$ and $P$ be the path joining $u_1$ and $u_2$ in $T$. Then
\begin{equation}\label{l6.4e1}
d_{w} (u_1, u_2) = \max_{v \in V(P)} l(v)
\end{equation}
holds.
\end{lemma}

\begin{proof}
Let $P = (v_0, v_1, \ldots, v_n, v_{n+1})$ and \(v_0 = u_1\), \(v_{n+1} = u_2\) hold. It can be shown (See Figure~\ref{fig4}) that there is $i^{*} \in \{1, \ldots, n\}$ such that:
\begin{itemize}
\item $v_i$ is a direct successor of $v_{i+1}$ if $i+1 \leqslant i^{*}$;
\item $v_{i+1}$ is a direct successor of $v_i$ if $i+1 > i^{*}$.
\end{itemize}
Since $\delta^{+}(u_1) = \delta^{+}(u_2) = 0$, it is clear also that $v_0$ is a direct successor of $v_1$ and $v_{n+1}$ is a direct successor of $v_n$. Now from~\eqref{e6.7} and~\eqref{p6.1e1} it follows that
\begin{multline}\label{l6.4e2}
2 d_w (u_1, u_2) = 2 d_w (v_0, v_{n+1})= (l(v_1) - l(v_0)) + \sum_{i=1}^{i^{*}-1} (l(v_{i+1}) - l(v_{i})) \\
{}+ \sum_{i=i^{*}}^{n-1} (l(v_{i}) - l(v_{i+1})) + (l(v_n) - l(v_{n+1})) = 2 l(v_{i^{*}}) - (l(u_1) + l(u_2)).
\end{multline}
The labeling $l = \hat{l}*w$ is monotone. Hence we have
\begin{equation}\label{l6.4e3}
l(v^{*}) = \max_{1 \leqslant i \leqslant n} l(v_i) \quad \text{and} \quad l(u_1) = l(u_2) = 0.
\end{equation}
Equality~\eqref{l6.4e1} follows from \eqref{l6.4e2} and \eqref{l6.4e3}.
\end{proof}

The dual form of Lemma~\ref{l6.4} can be formulated as follows (cf. Theorem~\ref{t2.9}).

\begin{lemma}\label{l3.4}
Let $T = T(r, l)$ be a monotone rooted tree, let $w = \hat{w} * l$ and let $u_1$, $u_2 \in V_{0}^{+}(T)$ and $u_1 \neq u_2$. Then~\eqref{l6.4e1} holds for the path $P$ joining $u_1$ and $u_2$ in~$T$.
\end{lemma}

It is well known that, for phylogenetic equidistant trees, the restriction of $d_{w}$ on the set of leaves of $T$ is an ultrametric (see Theorem~7.2.5 in~\cite{SS2003OUP}). Lemma~\ref{l6.4} and Proposition~\ref{p3.5} imply the following generalization of this result.

\begin{theorem}\label{t6.4}
Let $T = T(w)$ be a weighted tree with strictly positive weight and let $d_{w} \colon V(T) \times V(T) \to \mathbb{R}^{+}$ be the corresponding additive metric. If there is $r \in V(T)$ such that $T(r, w)$ is an equidistant tree, then the restriction $d_{w}|_{V_{0}^{+}(T) \times V_{0}^{+}(T)}$ is an ultrametric on \(V_{0}^{+}(T)\).
\end{theorem}

\begin{figure}[h]
\begin{center}
\begin{tikzpicture}[node distance=1cm,on grid]
\tikzset{solid/.style={circle,draw,inner sep=1.5,fill=black}}

\node [solid, label=left:$x_1$] (x1) {};
\node [solid, label=left:$x_2$] (x2) [below=of x1] {};
\node [solid] (y1) [right=of x1, yshift=-0.5cm] {};
\node [solid] (y2) [right=of y1] {};
\node [solid, label=right:$x_3$] (x3) [right=of y2, yshift=0.5cm] {};
\node [solid, label=right:$x_4$] (x4) [below=of x3] {};
\node [right=of x3, label=right:$T(w)$] {};
\draw (x1)-- node [above] {$2$} (y1)-- node [above] {$3$} (y2)-- node [above] {$1$} (x3);
\draw (y2)-- node [below] {$1$} (x4);
\draw (x2)-- node [below] {$2$} (y1);
\end{tikzpicture}
\end{center}
\caption{The restriction of $d_{w}$ on the set $X = \{x_1, x_2, x_3, x_4\}$ of the leaves of $T(w)$ is an ultrametric.}
\label{fig13}
\end{figure}
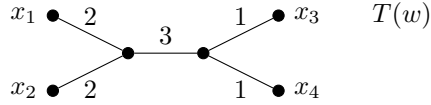

There is a weighted tree $T = T(w)$ such that the restriction of additive metric $d_{w}$ on the set $L = L(T)$ of all leaves of $T$ is an ultrametric but the tree $T(r, w)$ is not equidistant for any choice of the root $r \in V(T)$. (See Figure~\ref{fig13}.) Since, for every $r \in V(T)$, we have
$$
V_{0}^{+}(T) \subseteq L(T),
$$
the restriction $d_{w}|_{V_{0}^{+}(T) \times V_{0}^{+}(T)}$ is also an ultrametric on $V_{0}^{+}(T)$. Thus, the converse theorem to Theorem~\ref{t6.4} does not hold in general.

Let \(T = T(r, w)\) be a weighted rooted tree. In what follows we will use the denotation
\begin{equation}\label{e3.10}
V_{1}^{+}(T) := \{v \in V(T) \colon \delta^{+} (v) = 1\}.
\end{equation}

\begin{proposition}\label{p3.8}
Let \(T = T(r, w)\) be an equidistant tree and let \(\delta^{+}(v) \geqslant 2\) hold for some \(v \in V(T)\). Then there is an equidistant tree \(T^{\nabla} = T^{\nabla}(r^{\nabla}, w^{\nabla})\) such that 
\begin{gather}
V_1^{+}(T^{\nabla}) = \varnothing, \quad V_0^{+}(T) = V_0^{+}(T^{\nabla}), \quad V(T^{\nabla}) \subseteq V(T),\notag \\
\label{p3.8:e1} |V(T^{\nabla})| + |V_1^{+}(T)| = |V(T)|,
\end{gather}
and
\[
d_{w}|_{V_0^{+}(T) \times V_0^{+}(T)} = d_{w^{\nabla}}|_{V_0^{+}(T^{\nabla}) \times V_0^{+}(T^{\nabla})}.
\]
\end{proposition}

\begin{proof}
The proposition is evident if \(|V_{0}^{+}(T)| = 1\) or \(V_{1}^{+}(T) = \varnothing\). Let \(V_{1}^{+}(T) \neq \varnothing\) and $|V_{0}^{+}(T)| \geqslant 2$ hold. The last inequality holds if and only if the set
$$
V_{2}^{++}(T) := \{v \in V(T) \colon \delta^{+}(v) \geqslant 2\}
$$
is nonempty. Since $\delta^{+}(r)$ is strictly positive, we have either $r \in V_{2}^{++}(T)$, or \(r \in V_{1}^{+}(T)\). We first do the case $r \in V_{2}^{++}(T)$. Starting from the tree $T(r, w)$ we define an equidistant tree $T^{\nabla} = T^{\nabla}(r^{\nabla}, w^{\nabla})$ by the following inductive rule. 

We choose an arbitrary $v^{*} \in V_{1}^{+}(T)$ and consider the weighted rooted tree $T_{1} = T_{1}(r_1, w_1)$ such that
\begin{gather}\label{t6.4e1}
r_1 = r \quad \text{and}\quad V(T_{1}) = V(T) \setminus \{v^*\} \quad\text{and}\\
\label{t6.4e2}
E(T_{1}) = \Bigl(E(T) \cup \bigl\{\{u_1, u_2\}\bigr\}\Bigr) \setminus \bigl\{\{u_1, v^{*}\}, \{u_2, v^{*}\}\bigr\},
\end{gather}
and
\begin{equation}\label{t6.4e3}
w_1(e) = \begin{cases}
w(e), & \text{if $e \in E(T)$}\\
w(\{u_1, v^{*}\}) + w(\{v^{*}, u_2\}), & \text{if $e = \{u_1, u_2\}$},
\end{cases}
\end{equation}
where $u_1$ and $u_2$ are the neighbors of $v^*$ in $T$. If $V_{1}^{+}(T_1) = \varnothing$, then we set 
\[
T^{\nabla}(r^{\nabla}, w^{\nabla}) = T_{1}(r_1, w_1).
\]
Otherwise, by repeating the above-described procedure with $T_{1}(r_1, w_1)$ instead of $T(r, w)$, we obtain the weighted rooted tree $T_{2}(r_2, w_2)$, etc. Since $V_1^{+}(T)$ is finite and 
\[
|V_1^{+}(T)| = |V_1^{+}(T_1)| + 1 = |V_1^{+}(T_2)| + 2 = \ldots
\]
holds, we have $V_1^{+}(T_k) = \varnothing$ for some $k \geqslant 1$. Thus $T^{\nabla}(r^{\nabla}, w^{\nabla}) = T_{k}(r_k, w_k)$ and the construction of $T^{\nabla}$ is completed.

It follows from~\eqref{t6.4e1}--\eqref{t6.4e3}, that 
\[
r^{\nabla} = r, \quad V(T^{\nabla}) = V_0^{+}(T) + V_{2}^{++}(T), \quad V_{1}^{+}(T^{\nabla}) = \varnothing
\]
and $V_0^{+}(T^{\nabla}) = V_0^{+}(T) = L(T^{\nabla})$, where $L(T^{\nabla})$ is the set of leaves of $T^{\nabla}$. Moreover, the equality 
\[
d_{w}(u^{\nabla}, v^{\nabla}) = d_{w^{\nabla}}(u^{\nabla}, v^{\nabla})
\]
holds for all $u^{\nabla}$, $v^{\nabla} \in V(T^{\nabla})$. Hence $T^{\nabla}(r^{\nabla}, w^{\nabla})$ is equidistant and
\begin{equation}\label{t6.4e4}
d_w|_{V_0^+(T) \times V_0^+(T)} = d_{w^{\nabla}}|_{V_0^+(T) \times V_0^+(T)} = d_{w^{\nabla}}|_{V_0^+(T^{\nabla}) \times V_0^+(T^{\nabla})}.
\end{equation}

In the case $r \in V_1^+(T)$ we define the weighted rooted tree $T_1(r_1, w_1)$ such that $r_1$ is the unique direct successor of $r$, 
\[
V(T_1) = V(T) \setminus \{r\}, \quad E(T_1) = E(T) \setminus \{\{r_1, r\}\}
\]
and $w_1 = w|_{E(T_1)}$. Repeating this procedure, we find the smallest positive integer $k_0$ such that $r_{k_0} \in V_2^{++}(T)$. The weighted rooted tree $T_{k_0} = T_{k_0}(r_{k_0}, w_{k_0})$ is equidistant and 
\[
d_w|_{V(T_{k_0}) \times V(T_{k_0})} = d_{w_{k_0}} \quad \text{and} \quad V_0^+(T) = V_0^+(T_{k_0}).
\]
Now using~\eqref{t6.4e1}--\eqref{t6.4e3} with $T_{k_0}(r_{k_0}, w_{k_0})$ instead of $T(r, w)$ we can construct \(T^{\nabla}(r^{\nabla}, w^{\nabla})\) such that \(V_1^{+}(T^{\nabla}) = \varnothing\) and \eqref{t6.4e4} hold.
\end{proof}

\begin{remark}\label{r3.9}
The condition
\begin{itemize}
\item there is \(v \in V(T)\) such that \(\delta^{+}(v) \geqslant 2\)
\end{itemize}
cannot be dropped in Proposition~\ref{p3.8}. If the inequality \(\delta^{+}(v) \leqslant 1\) holds for all \(v \in V(T)\), then \(T = T(r, w)\) is a weighted path joining the root \(r\) with the unique vertex belonging to the one-point set \(V_0^{+}(T)\). In this case we have the equality
\[
|V_1^{+}(T)| + 1 = |V(T)|,
\]
which together with \eqref{p3.8:e1} implies \(|V(T^{\nabla})| = 1\). Thus, \(T^{\nabla}\) is empty, contrary to Definition~\ref{d2.2}
\end{remark}

An example of transition from \(T(r, w)\) to \(T^{\nabla}(r^{\nabla}, w^{\nabla})\) is given by Figure~\ref{fig5}.

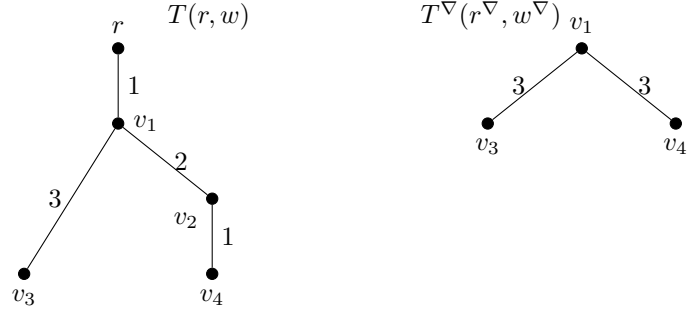
\begin{figure}[h]
\begin{center}
\begin{tikzpicture}[level 1/.style={level distance=1cm,sibling distance=25mm},
level 2/.style={level distance=1cm,sibling distance=25mm},
level 3/.style={level distance=1cm,sibling distance=25mm},
solid node/.style={circle,draw,inner sep=1.5,fill=black},
hollow node/.style={circle,draw,inner sep=1.5}]
\node (root1) [solid node, label=above:{\(r\)}] at (0,0) {}
	child{node[solid node, label=right:{\(v_1\)}]{}
		child [level distance=2cm]{node[solid node, label=below:{\(v_3\)}]{} edge from parent node[left]{$3$}}
		child{node[solid node, label=below left:{\(v_2\)}]{} 
			child{node[solid node, label=below:{\(v_4\)}]{} edge from parent node[right]{$1$}}
			edge from parent node[right]{$2$}}
		edge from parent node[right]{$1$}
	};
\node (title1) [right=of root1, label=above:{\(T(r, w)\)}] {};

\node(root2)[solid node, right=of root1, xshift=5cm, label=above:{\(v_1\)}] {}
	child {node [solid node, label=below:{\(v_3\)}]{} edge from parent node[left]{$3$}}
	child {node [solid node, label=below:{\(v_4\)}]{} edge from parent node[right]{$3$}};
\node (title2) [left=of root2, label=above:{\(T^{\nabla}(r^{\nabla}, w^{\nabla})\)}] {};
\end{tikzpicture}
\end{center}
\caption{Here \(T(r, w)\) and \(T^{\nabla}(r^{\nabla}, w^{\nabla})\) are equidistant, \(r^{\nabla} = v_1\) and \(w^{\nabla}(\{u, v\}) = d_w(u, v)\) holds for every \(\{u, v\} \in E(T^{\nabla})\).}
\label{fig5}
\end{figure}

\begin{remark}\label{r4.10}
The set \(V(T^{\nabla})\) has the following nice geometric interpretation:
\begin{itemize}
\item A point \(x\) of the metric space \((V(T), d_{w})\) belongs to \(V(T^{\nabla})\) if and only if 
\[
d_{w}(y,x) = d_{w}(x,z) = \frac{1}{2} d_{w}(y,z)
\]
holds for some \(y\), \(z \in V_0^{+}(T)\).
\end{itemize}
\end{remark}

The following proposition can be considered as a generalization of Theorem~7.2.8 from~\cite{SS2003OUP}.

\begin{proposition}\label{p6.9}
Let $T_1(r_1, w_1)$ and $T_2(r_2, w_2)$ be equidistant trees. Then
\begin{multline*}
\Bigl(T_1^{\nabla}\bigl(r_1^{\nabla}, w_1^{\nabla}\bigr) \simeq T_2^{\nabla}\bigl(r_2^{\nabla}, w_2^{\nabla}\bigr)\Bigr) \Leftrightarrow \\
\left(\left(V_0^{+}(T_1), d_{w_1}|_{V_0^{+}(T_1) \times V_0^{+}(T_1)}\right) \simeq \left(V_0^{+}(T_2), d_{w_2}|_{V_0^{+}(T_2) \times V_0^{+}(T_2)}\right)\right)
\end{multline*}
is valid
\end{proposition}

\begin{proof}
It follows from Proposition~\ref{p3.3} and Theorem~\ref{t2.7}.
\end{proof}

\begin{theorem}\label{c3.13}
Let \(T = T(r)\) be a nonempty rooted tree. Then the following statements are equivalent:
\begin{enumerate}
\item\label{c3.13:s1} The set \(V_1^{+}(T)\) is empty.
\item\label{c3.13:s2} The logical equivalence
\begin{multline*}
\bigl(T(r, w_1) \simeq T(r, w_2)\bigr) \Leftrightarrow\\
\left(\left(V_0^{+}(T), d_{w_1}|_{V_0^{+}(T) \times V_0^{+}(T)}\right) \simeq \left(V_0^{+}(T), d_{w_2}|_{V_0^{+}(T) \times V_0^{+}(T)}\right)\right)
\end{multline*}
is valid whenever \(w_1 \colon E(T) \to \mathbb{R}^{+}\) and \(w_2 \colon E(T) \to \mathbb{R}^{+}\) are equidistant.
\end{enumerate}
\end{theorem}

\begin{proof}
\(\ref{c3.13:s1} \Rightarrow \ref{c3.13:s2}\). Let \ref{c3.13:s1} hold. Then we have the equality
\[
T^{\nabla}(r^{\nabla}, w_i^{\nabla}) = T^{\nabla}(r, w_i^{\nabla})
\]
for \(i = 1\), \(2\) and the validity of \ref{c3.13:s2} follows from Proposition~\ref{p6.9}.

\(\ref{c3.13:s2} \Rightarrow \ref{c3.13:s1}\). This implication is true if and only if 
\[
{\rceil}\ref{c3.13:s1} \Rightarrow {\rceil}\ref{c3.13:s2},
\]
where \({\rceil}\) is the logical negation symbol. Suppose there is \(v^{*} \in V^{+}(T)\) such that \(\delta^{+}(v^{*}) = 1\). If \(v^{*} \neq r\), then we have \(\delta(v^{*}) = 1\) and, consequently, there are exactly two nodes \(v_1\) and \(v_2\) such that \(v^{*}\) is a direct successor of \(v_1\) and \(v_2\) is the unique direct successor of \(v^{*}\).

Let \(w \colon E(T) \to \mathbb{R}^{+}\) be an equidistant weight. Then we can find strictly positive, pairwise distinct real numbers \(t_1\), \(t_2\), \(s_1\), \(s_2\) such that
\begin{equation}\label{c3.13:e2}
t_1 + t_2 = s_1 + s_2 = w(\{v_1, v^{*}\}) + w(\{v^{*}, v_2\})
\end{equation}
and
\begin{equation}\label{c3.13:e3}
(w(e) - t_1)(w(e) - t_2)(w(e) - s_1)(w(e) - s_2) \neq 0
\end{equation}
holds for every \(e \in E(T)\).

Let us define the weights \(w_1 \colon E(T) \to \mathbb{R}^{+}\) and \(w_2 \colon E(T) \to \mathbb{R}^{+}\) as
\[
w_1(e) = \begin{cases}
s_1, & \text{if } e = \{v_1, v^{*}\},\\
s_2, & \text{if } e = \{v^{*}, v_2\},\\
w(e), & \text{otherwise}
\end{cases} \quad \text{and} \quad
w_2(e) = \begin{cases}
t_1, & \text{if } e = \{v_1, v^{*}\},\\
t_2, & \text{if } e = \{v^{*}, v_2\},\\
w(e), & \text{otherwise}.
\end{cases}
\]
Then \(w_1\) and \(w_2\) are equidistant and, moreover, using formula \eqref{t6.4e3} we obtain the equalities
\[
w^{\nabla} = w_1^{\nabla} = w_2^{\nabla}.
\]
Now Proposition~\ref{p3.8} implies
\[
d_{w_1}|_{V_0^{+}(T) \times V_0^{+}(T)} = d_{w_2}|_{V_0^{+}(T) \times V_0^{+}(T)}.
\]
Since \(t_1\), \(t_2\), \(s_1\), \(s_2\) are pairwise distinct, from condition~\eqref{c3.13:e3} it follows that \(T(r, w_1)\) and \(T(r, w_2)\) are not isomorphic. Thus, \ref{c3.13:s2} is false. Hence, \({\rceil}\ref{c3.13:s1} \Rightarrow {\rceil}\ref{c3.13:s2}\) is valid if \(v^{*} \neq r\). 

The case \(v^{*} = r\) is more simple and can be considered similarly.
\end{proof}

The following result is a dual form of Theorem~\ref{c3.13} and it is a partial generalization of Theorem~\ref{t2.7}.

\begin{theorem}\label{t4.13}
Let \(T = T(r)\) be a rooted tree. Then the following statements are equivalent:
\begin{enumerate}
\item\label{t4.13:s1} The set \(V_1^{+}(T)\) is empty.
\item\label{t4.13:s2} The logical equivalence
\begin{multline*}
\bigl(T(r, l_1) \simeq T(r, l_2)\bigr) \Leftrightarrow\\
\left(\left(V_0^{+}(T), d_{l_1}|_{V_0^{+}(T) \times V_0^{+}(T)}\right) \simeq \left(V_0^{+}(T), d_{l_2}|_{V_0^{+}(T) \times V_0^{+}(T)}\right)\right)
\end{multline*}
is valid whenever \(l_1 \colon V(T) \to \mathbb{R}^{+}\) and \(l_2 \colon V(T) \to \mathbb{R}^{+}\) are monotone.
\end{enumerate}
\end{theorem}

\begin{proposition}\label{p6.10}
Let $(X, d)$ be a finite ultrametric space with $|X| \geqslant 2$, let $T_{X}$ be the representing tree of $(X, d)$ and let $T = T(r,w)$ be an equidistant tree such that $(V_0^{+}(T), d_{w}|_{V_0^{+}(T) \times V_0^{+}(T)}) \simeq (X, d)$. Then the inequality
\begin{equation}\label{p6.10e1}
|V(T)| \geqslant |\mathbf{B}_X|
\end{equation}
holds, where $\mathbf{B}_X$ is the set of balls of $(X, d)$. Moreover, the equality $|V(T)| = |\mathbf{B}_X|$ holds if and only if $T_X \simeq T(r,w)$.
\end{proposition}

\begin{proof}
Let us consider $l \colon V(T_X) \to \mathbb{R}^{+}$ defined by~\eqref{e2.7}. By Proposition~\ref{p6.1}, the weight $\hat{w}*l \colon E(T_X) \to \mathbb{R}^{+}$ is equidistant. Using Lemma~\ref{l6.4} we obtain that $(X, d) \simeq (V_0^{+}(T), d_{\hat{w}*l}|_{V_0^{+}(T_{X}) \times V_0^{+}(T_{X})})$. Consequently, 
\[
\left(V_0^{+}(T_{X}), d_{\hat{w}*l}|_{V_0^{+}(T_{X}) \times V_0^{+}(T_{X})}\right) \simeq \left(V_0^{+}(T), d_{w}|_{V_0^{+}(T) \times V_0^{+}(T)}\right)
\]
is valid. Now from Proposition~\ref{p6.9} it follows that 
\begin{equation}\label{p6.10e2}
T^{\nabla} (r^{\nabla}, w^{\nabla}) = T_X(X, \hat{w}*l).
\end{equation}
(Note that $T^{\nabla}_X = T_X$ holds because $V_0^{+}(T_X) = \varnothing$). Since 
\[
V(T^{\nabla}, r^{\nabla}, w^{\nabla}) \subseteq V(T, r, w)
\]
and $|V(T_{X})| = |\mathbf{B}_{X}|$ (see Proposition~\ref{p2.8}), inequality~\eqref{p6.10e1} follows.

Suppose that $T_X \simeq T(r, w)$. Then we have 
\[
|V(T(r, w))| = |V(T_X)| = |\mathbf{B}_X|.
\]
Now, to complete the prove note that $|V(T_{X})| = |\mathbf{B}_{X}|$ holds if and only if
\[
T^{\nabla}(r^{\nabla}, w^{\nabla}) = T(r,w),
\]
so the describe isomorphism of $T_X$ and $T(r,w)$ as rooted trees follows from~\eqref{p6.10e2}.
\end{proof}

\section{Planted equidistant trees and ultrametrics}\label{sec5}

Recall that a rooted tree $T = T(r)$ is planted if $\delta^{+}(r) = 1$ holds.

The following propositions clarify the ``ultrametric'' meaning of the constant \(K\) from the definition of equidistant trees (see formula~\eqref{e6.1}).

\begin{proposition}\label{p4.1}
Let \(T = T(r, w)\) be an equidistant tree and let \(d_1\) be the restriction of the additive metric \(d_w\) on the set \(V_1^{+}(T) \times V_1^{+}(T)\). Then the diameter \(\diam(V_1^{+}(T))\) of ultrametric space \((V_1^{+}(T), d_1)\) satisfies the inequality
\begin{equation}\label{p4.1:e1}
\diam(V_1^{+}(T)) \leq 2K.
\end{equation}
This inequality is strict if and only if \(T\) is planted.
\end{proposition}

\begin{proposition}\label{p6.11}
Let $T = T(w)$ be a weighted tree with \(|V(T)| \geqslant 2\) and let $L = L(T)$ be the set of leaves of $T$. Suppose the restriction $\rho := d_{w} |_{L \times L}$ is an ultrametric on $L$. Then the following conditions are equivalent:
\begin{enumerate}
\item There is $r \in L$ such that the weighted rooted tree $T(r, w)$ is equidistant.
\item The ultrametric space $(L, \rho)$ is a sphere with an added center, i.e., there are $c \in L$ and $t > 0$ such that the equality $L = S_{t}(c) \cup \{c\}$ holds, where $S_{t}(c) := \{x \in L \colon \rho(x, c) = t\}$.
\item There is $r \in V(T)$ such that $T(r, w)$ is planted and equidistant.
\end{enumerate}
\end{proposition}

The proofs of these propositions are straightforward and we omit it here.

It should be noted that for some planted equidistant trees $T = T(r, w)$ the restriction $d_{w} |_{L \times L}$ on the set $L$ of the leaves of $T$ is not an ultrametric (see Figure~\ref{fig10}). The following proposition describes the geometry of planted equidistant trees $T(r, w)$ for which $d_{w} |_{L \times L}$ is an ultrametric.

\begin{proposition}\label{p6.12}
Let $T = T(r,w)$ be planted and equidistant and let
\begin{equation*}
V_{2}^{++}(T) := \{v \in V(T) \colon \delta^{+}(v) \geqslant 2\} \neq \varnothing.
\end{equation*}
Then the following conditions are equivalent:
\begin{enumerate}
\item\label{p6.12s1} The metric $\rho = d_{w} |_{L \times L}$ is an ultrametric on $L$.
\item\label{p6.12s2} The inequality 
\begin{equation}\label{p6.12e2}
2\operatorname{dist} (r, V_{2}^{++}(T)) \geqslant \operatorname{dist} (r, V_{0}^{+}(T))
\end{equation}
holds, where, for every nonempty $A \subset V(T)$,
$$
\operatorname{dist}(r, A) := \min\{d_{w}(r, a) \colon a \in A\}.
$$
\end{enumerate}
\end{proposition}

\begin{proof}
$\ref{p6.12s1} \Rightarrow \ref{p6.12s2}$ Suppose~\ref{p6.12s1} holds. We must prove inequality~\eqref{p6.12e2}. Let $v^* \in V_{2}^{++}(T)$ such that 
\begin{equation}\label{p6.12e3}
d_{w}(r, v^*) = \operatorname{dist}(r, V_{2}^{++}(T)).
\end{equation}
Let $P = (v_0, \ldots, v_n)$ be the path in $T$ with $v_0 = r$ and $v_n = v^*$. Since $P$ is a path, the inequality $\delta^{+} (v_i) \geqslant 1$ holds for every $i \in \{1, \ldots, n-1\}$. If there is $i_{1} \in \{1, \ldots, n-1\}$ such that $\delta^{+}(v_{i_1}) \geqslant 2$, then $v_{i_1} \in V_{2}^{++}(T)$ and we have
\begin{multline*}
d_{w} (r, v^{*}) = \sum_{j = 0}^{n-1} w(\{v_{j}, v_{j+1}\}) = \sum_{j = 0}^{i_1-1} w(\{v_{j}, v_{j+1}\}) + \sum_{j = i_1}^{n-1} w(\{v_{j}, v_{j+1}\}) \\
= d_{w} (r, v_{i_1}) + d_{w} (v_{i_1}, v^{*}) > d_{w} (r, v_{i_1}),
\end{multline*}
contrary to~\eqref{p6.12e3}.

Inequality $\delta^{+}(v^{*}) \geqslant 2$ implies that there exist two distinct direct successors $u_1$ and $u_2$ of $v^{*}$. Let $x_i \in V_{0}^{+}(T)$ be a successor of $u_i$, $i=1$, $2$. It follows from the definition of $d_{w}$ that
\begin{equation}\label{p6.12e4}
d_{w}(r, x_i) = d_{w}(r, v^{*}) + d_{w}(v^{*}, x_i)
\end{equation}
for $i=1$, $2$, and, in addition,
\begin{equation}\label{p6.12e5}
d_{w}(x_1, x_2) = d_{w}(x_1, v^{*}) + d_{w}(v^{*}, x_2).
\end{equation}
Since $T(r, w)$ is equidistant, we have
\begin{equation}\label{p6.12e6}
d_{w}(r, x_1) = d_{w}(r, x_2) = \operatorname{dist}(r, V_{0}^{+}(T)).
\end{equation}
Equalities~\eqref{p6.12e4}, \eqref{p6.12e5} and \eqref{p6.12e6} imply 
\begin{equation}\label{p6.12e7}
d_{w}(x_1, x_2) = 2(\operatorname{dist}(r, V_{0}^{+}(T)) - d_{w}(r, v^{*})).
\end{equation}
From the ultrametricity $(L, \rho)$ and~\eqref{p6.12e6} it follows that
\begin{equation}\label{p6.12e8}
d_{w}(x_1, x_2) \leqslant \max_{i = 1, 2} d_{w}(r, x_i) = \operatorname{dist}(r, V_{0}^{+}(T)).
\end{equation}
Now using~\eqref{p6.12e3}, \eqref{p6.12e7} and~\eqref{p6.12e8} we obtain
$$
d_{w}(x_1, x_2) = 2(\operatorname{dist}(r, V_{0}^{+}(T)) -\operatorname{dist}(r, V_{2}^{++}(T))) \leqslant \operatorname{dist}(r, V_{0}^{+}(T)).
$$
Inequality~\eqref{p6.12e2} follows.

$\ref{p6.12s2} \Rightarrow \ref{p6.12s1}$ Let inequality~\eqref{p6.12e2} hold. By Theorem~\ref{t6.4} the metric $d_{w}|_{V_{0}^{+}(T) \times V_{0}^{+}(T)}$ is an ultrametric. Consequently $\rho = d_{w}|_{L(T) \times L(T)}$ is an ultrametric on $L(T)$ if and only if the strong triangle inequalities
\begin{equation}\label{p6.12e9}
d_{w} (x_1, x_2) \leqslant \max \{d_{w} (r, x_1), d_{w} (r, x_2)\}
\end{equation}
and 
\begin{equation}\label{p6.12e10}
d_{w} (r, x_1) \leqslant \max \{d_{w} (x_1, x_2), d_{w} (r, x_2)\}
\end{equation}
and 
\begin{equation}\label{p6.12e10*}
d_{w} (r, x_2) \leqslant \max\{d_{w} (x_1, x_2), d_{w} (r, x_1)\}
\end{equation}
hold. Inequalities~\eqref{p6.12e10} and~\eqref{p6.12e10*} are trivial because $T(r, w)$ is equidistant. Let $v^{*}$ be a point of $V_{2}^{++}(T)$ satisfying~\eqref{p6.12e3}. The triangle inequality implies
\begin{equation}\label{p6.12e11}
d_{w} (x_1, x_2) \leqslant d_{w} (x_1, v^{*}) + d_{w} (v^{*}, x_2).
\end{equation}
The induced rooted subtree $T_{v^{*}}$ of $T(r)$ is equidistant with the weight $w|_{E(T_{v^{*}})}$. It implies
\begin{multline}\label{p6.12e12}
d_{w} (x_1, v^{*}) + d_{w} (v^{*}, x_2) = 2 \min_{v \in V_{0}^{+}(T_{v^*})} d_{w} (v^{*}, v) \\
= 2 (\min_{v \in V_{0}^{+}(T(r))} d_{w} (r, v) - d_w(r, v^{*})) = 2(\operatorname{dist}(r, V_{0}^{+}(T)) -\operatorname{dist}(r, V_{2}^{++}(T)))
\end{multline}
Moreover, we have
\begin{equation}\label{p6.12e13}
\max\{d_{w} (r, x_1), d_{w} (r, x_2)\} = \min_{v \in V_{0}^{+}(T(r))} d_{w} (r, v) = \operatorname{dist}(r, V_{0}^{+}(T)).
\end{equation}
Consequently it suffices to show that 
$$
2(\operatorname{dist}(r, V_{0}^{+}(T)) -\operatorname{dist}(r, V_{2}^{++}(T))) \leqslant \operatorname{dist}(r, V_{0}^{+}(T)).
$$
The last inequality is equivalent to inequality~\eqref{p6.12e2}.
\end{proof}

Let \(T = T(w)\) be a weighted tree. We say that a node \(v^{*} \in V(T)\) is a \emph{center} of \(T\) if the rooted tree \(T(r, w)\) is equidistant with \(r = v^{*}\). The following example (see Figure~\ref{fig3}) shows that several distinct nodes of \(T\) can be centers of \(T\).

\begin{figure}[h]
\begin{center}
\begin{tikzpicture}
\tikzstyle{level 1}=[level distance=15mm,sibling distance=15mm]
\tikzstyle{level 2}=[level distance=15mm,sibling distance=10mm]
\tikzstyle{level 3}=[level distance=15mm,sibling distance=6mm]
\tikzset{
solid node/.style={circle,draw,inner sep=1.5,fill=black},
hollow node/.style={circle,draw,inner sep=1.5}
}

\node [label=left:{\(T(w)\)}] at (2,0) {};
\node (1) [solid node, label=above:{\(v_1\)}] at (0,0) {}
child {node[solid node, label=above:{\(v_2\)}]{}
	child{node[solid node, label=below:{\(v_{6}\)}] {} edge from parent node[left]{$2$}}
	child{node[solid node, label=below:{\(v_{7}\)}] {} edge from parent node[right]{$2$}}
	edge from parent node[above left]{$2$}}
child {node[solid node, label=above left:{\(v_3\)}]{}
	child{node[solid node, label=below:{\(v_{8}\)}]{} edge from parent node[left]{$2$}}
	child{node[solid node, label=below:{\(v_{9}\)}]{} edge from parent node[right]{$2$}}
	edge from parent node[left]{$2$}}
child {node[solid node, label=below:{\(v_4\)}]{} edge from parent node[left]{$4$}}
child {node[solid node, label=below:{\(v_5\)}]{} edge from parent node[above right]{$4$}};
\end{tikzpicture}
\end{center}
\caption{\(T(r,w)\) is equidistant if and only if \(r \in \{v_1, v_4, v_5\}\).}
\label{fig3}
\end{figure}
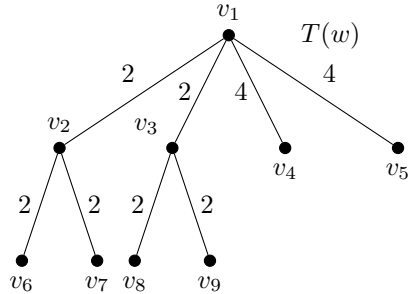

\begin{proposition}\label{p3.9}
Let \(T\) be a nonempty tree. Then the following statements are equivalent:
\begin{enumerate}
\item\label{p3.9:s1} \(T\) is star.
\item\label{p3.9:s2} There is a strictly positive weight \(w \colon E(T) \to \mathbb{R}^{+}\) such that the weighted rooted tree \(T(r, w)\) is equidistant for every \(r \in V(T)\).
\end{enumerate}
\end{proposition}

\begin{proof}
\(\ref{p3.9:s1} \Rightarrow \ref{p3.9:s2}\). Let \(T\) be a star and let \(c\) be a strictly positive real number. Then \(T(r, w)\) is equidistant for every \(r \in V(T)\) if we define \(w \colon E(T) \to \mathbb{R}^{+}\) as
\[
w(e) = c
\]
for every \(e \in E(T)\).

\(\ref{p3.9:s2} \Rightarrow \ref{p3.9:s1}\). Let statement \ref{p3.9:s2} hold. It is clear that \ref{p3.9:s2} holds if \(|V(G)| = 2\). Suppose \(|V(G)| \geqslant 3\). Then there is \(v^{*} \in V(T)\) such that \(\delta(v^{*}) \geqslant 2\) holds. We claim that, for every \(u \in V(T)\), the inequality \(\delta(u) \geqslant 2\) implies the equality \(v^{*} = u\). Indeed, if \(u \neq v^{*}\) and \(u\), \(v^{*} \in V(T)\), then there is a path \(P_0 = \{v_1, \ldots, v_m\}\) in \(T\) with \(v_1 = u\) and \(v_m = v^{*}\). The inequalities \(\delta(u) \geqslant 2\) and \(\delta(v^{*}) \geqslant 2\) imply that there is \(v_{-1}\) and \(v_{m+1}\) such that \(v_{-1}\), \(v_{m+1} \notin V(P_0)\) and 
\[
\{v_{-1}, v_1\}, \{v_{m}, v_{m+1}\} \in E(T).
\]
Since \(T\) is an acyclic graph, we have \(v_{-1} \neq v_{m+1}\). Consequently, 
\[
P_1 = \{v_{-1}, v_1, \ldots, v_m, v_{m+1}\}
\] 
is a path in \(T\) and \(P_1 \supseteq P_0\). If \(\delta(v_{-1}) \geqslant 2\) or \(\delta(v_m+1) \geqslant 2\) holds, then we can find a path \(P_2 \supseteq P_1\) such that \(P_2 \subseteq T\) and \(V(P_1)\) is a proper subset of \(V(P_2)\). Since \(T\) is a finite tree, there is a path \(P \supseteq P_0\), \(P \subseteq T\), joining some leaves \(a\), \(b \in L(T)\),
\[
P = (a, \ldots, u \ldots, v^{*}, \ldots, b).
\]
The rooted trees \(T(u, w)\) and  \(T(v^{*}, w)\) are equidistant. Hence, the equalities
\[
d_w(a, u) = d_w(u, b) \quad \text{and} \quad d_w(a, v^{*}) = d_w(v^{*}, b)
\]
hold. Using these equalities and the definition of additive metric \(d_w\) we obtain
\begin{multline*}
d_w(a, u) + d_w(v^{*}, b) = d_w(u, b) + d_w(a, v^{*}) \\
= (d_w(u, v^{*}) + d_w(v^{*}, b)) + (d_w(a, u) + d_w(u, v^{*})).
\end{multline*}
That implies \(d(u, v^{*}) = 0\). Hence, \(u = v^{*}\) holds, contrary to \(u \neq v^{*}\). 

A connected graph \(G\) is a tree if and only if
\begin{equation}\label{p3.9:e1}
|V(G)| = |E(G)| + 1
\end{equation}
(\cite[Corollary~1.5.3]{Die2005}). Moreover, for every graph \(G\), we have
\begin{equation}\label{p3.9:e2}
\sum_{v \in V(G)} \delta(v) = 2|E(G)|
\end{equation}
(\cite[Theorem~1.1]{BM}). 

It was shown above that \(\delta(v) = 1\) holds for every \(v \in V(T)\) whenever \(v \neq v^{*}\). Consequently, using~\eqref{p3.9:e1} we can rewrite \eqref{p3.9:e2} as
\[
\delta(v^{*}) + |V(G)| - 1 = 2 (|V(G)| - 1).
\]
Thus, 
\[
\delta(v^{*}) = |V(G)| - 1
\]
holds. Consequently, the vertexes \(v^{*}\), \(v\) are adjacent for every \(v \neq v^{*}\) and \(\{v_1, v_2\} \notin E(T)\) if \(v_1 \neq v^{*}\) and \(v_2 \neq v^{*}\). Thus, \(T\) is a star.
\end{proof}

\begin{remark}\label{r3.10}
Other curious characterizations of stars are given by Corollary~4.9 in \cite{DMV2013AC} and by Corollary~8 in~\cite{DP2013SM}. These characterizations as well as Proposition~\ref{p3.9} describe some extremal properties of weighted stars.
\end{remark}

The results of this paper have some natural analogies in the case when weights and labelings of rooted trees are some functions whose range is the positive cone $E_+$ of an ordered vector space $E$ and the ultrametrics are replaced by some ``generalized ultrametrics'' taking their values in $E_+$.

\section*{Funding}
This investigation was partially supported in the frame of the project: Development of Mathematical Models, Numerical and Analytical Methods, and Algorithms for Solving Modern Problems of Biomedical Research. State registration number: 0117U002165.


\end{document}